\documentclass{article}

\usepackage{amsmath}
\usepackage{amsthm}
\usepackage{amssymb}

\newtheorem{theorem}{Theorem}
\newtheorem{lemma}{Lemma}
\newtheorem{corollary}{Corollary}
\theoremstyle{definition}
\newtheorem{example}{Example}
\newcommand{\la}{\lambda}
\newcommand{\bbN}{\mathbb{N}}

\newcommand{\M}{{\mathcal M}}
\newcommand{\Ga}{{\Gamma}}
\newtheorem{definition}{Definition}
\title{ $H^\infty$-calculus for   semigroup generators on BMO } 

\author{Tim Ferguson\ \ \ Tao Mei\footnote{Research partially supported by the NSF grants DMS-1632435 and DMS-1700171.}\ \ \ \ Brian Simanek}
\begin{document}\maketitle
\begin{abstract} We prove that the  negative generator $L$ of a   semigroup of positive contractions  on $L^\infty$ has bounded $H^\infty(S_\eta)$-calculus on the associated Poisson semigroup-BMO space   for any angle $\eta>\pi/2$, provided  $L$ satisfies Bakry-\'Emery's $\Gamma^2\geq0 $ criterion.  Our arguments  only rely on the properties of the underlying semigroup and works well in  the noncommutative setting.  A key ingredient of our argument is a type of quasi monotone properties for the subordinated  semigroup  $T_{t,\alpha}=e^{-tL^\alpha},0<\alpha<1$, that is proved in the first part of this article.    
\end{abstract}

\maketitle 
\section*{Introduction}
Let $\Delta=-\partial_x^2$ be the negative Laplacian operator on ${\Bbb R}^n$.  The associated  Poisson semigroup of operators  $P_t=e^{-t\sqrt\Delta}, t\geq0$ has many nice properties that make it  a very useful tool in the classical analysis. In particular,  the Poisson semigroup has a quasi monotone property that there exist constants $c_{r,j}$ such that, for any nonnegative function $f\in L^1({\Bbb R}^n,\frac1{1+|x|^2}dx)$, 
 \begin{eqnarray}\label{quasi}
   |t^j\partial^j_t P_tf|\leq c_{r,j}P_{rt}f,
 \end{eqnarray}
 for any $0<r<1,j=0,1,2,...$.
 As a first result of this article, we show that the quasi monotone property (\ref{quasi}) extends to all subordinated semigroups  $T_{t,\alpha}=e^{-t  L^\alpha} $ for all $0<\alpha<1$ if $L$ generates a   semigroup of positive preserving  operators on a Banach lattice $X$. The case of $0<\alpha\leq\frac12$ is easy and is previously known because of a precise subordination formula (see e.g. \cite{M08,JM12}). This type of quasi-monotonicity has been a useful tool in proving certain functional inequalities  (see  \cite{GJL19, M08,JM12, JM10}).
 
Functional calculus is a theory of studying functions of operators. The so-called $H^\infty$-calculus  is a generalization of  the Riesz-Dunford analytic functional calculus and defines $\Phi(L)$ via a Cauchy-type integral for an (unbounded) sectorial operator $L$   and a function $\Phi$ that is  bounded and holomorphic  in a sector $S_\eta$   of the complex plane. 
  $L$ is said to have the bounded $H^\infty$-calculus property if the so-defined $\Phi(L)$ extend to bounded operators on $X$  and  $\|\Phi(L)\| \leq c\|\Phi\|_\infty$ for all such $\Phi$'s.
The theory of bounded $H^\infty$-calculus has developed rapidly in the last thirty years  with many applications and interactions with   harmonic analysis, Banach space theory, and the theory of evolution equations, starting with A. McIntosh's seminal work in 1986 (\cite{Mc86}, \cite{Ha06},\cite{KW04},\cite{We01}).  

It is a major task  in the study of the bounded $H^\infty$-calculus theory  to determine which operators have such a strong property.   Cowling, Duong, and Hieber $\&$ Pr\"uss (\cite{Co83, Du89,HP98}) prove that the infinitesimal generator of a semigroup of positive contractions on $L^p, 1<p<\infty$ always has the bounded $H^\infty(S_\eta)$-calculus on $L^p$ for any $\eta>\frac\pi2$. When the semigroup is symmetric, the angle can be reduced to $\eta>\omega_p=|\frac\pi2-\frac \pi p|$ by interpolation. It is not surprising that this result fails for  $L^\infty$ in general.  One may want to seek a BMO-type space that could be  an appropriate alternative for the $p=\infty$ case. 

The main theorem of this article states that  the negative generator $L$ of a   semigroup of positive contractions   on $L^\infty$ always has bounded $H^\infty(S_\eta)$-calculus on the space BMO$(\sqrt L)$ for any $\eta>\frac\pi2$,   provided  $L$ satisfies Bakry-\'Emery's $\Gamma^2 $ criterion. 
Junge and Mei attempted to prove this result  (see Theorem 3.3 of  \cite{JM12})  under the same assumptions, but only managed to obtain a  bounded $H^\infty(S_\eta)$ ($\eta>\frac\pi2$) calculus result  for $\sqrt L$, instead of $L$. This is due to the fact that  Lemma 3.2 and Theorem 3.3 of \cite{JM12} are proved only for the operator $M_a$ defined for the subordinated Poisson semigroup $P_t=e^{-t\sqrt L}$.   The unknownness of the quasi-monotonicity for general subordinated semigroups $e^{-t  L^\alpha}$ was a major obstacle that prevented Junge and Mei  from reaching further. 
Please note that $L$ is incorrectly written in place of $\sqrt{L}$ in the proof of Corollary 5.4 in \cite{JM12}. Its corrected version is proved  in this article as Corollary \ref{correctLis}. 

The classical BMO norm of a  function $f\in L^1({\Bbb R}^n,\frac1{1+|x|^2}dx)$ can be defined as 
\begin{eqnarray}\label{BMOPoi}
\|f\|_{BMO(\sqrt \Delta)}=\sup_{t>0}  \left\| e^{-t\sqrt\Delta}\left|f-e^{-t\sqrt\Delta}f\right|^2\right\|_{L^\infty}^\frac12.
\end{eqnarray}
BMO spaces associated with semigroup generators have been intensively studied recently (e.g. \cite{DY05} and the subsequent works ). When a cubic-BMO
is available, one  can often compare it with the semigroup BMO  and they are equivalent in many cases.
 In this article, we   consider the BMO$(\sqrt L)$-(semi)norm  studied in \cite{JM12,M08}, which are defined similarly to  (\ref{BMOPoi}), merely replacing  $\Delta$ with the  semigroup generator $L$. The corresponding space BMO$(\sqrt L)$ 
 interpolates well with $L^p$-spaces when the semigroup is symmetric Markovian (see Lemma \ref{JM12}). 

Under the assumptions of our main theorem, we also study semigroup-BMO spaces $BMO(L^\alpha), 0<\alpha<1$ and prove that they are all equivalent. We further prove that the imaginary power $L^{is}$ is bounded on the associated semigroup-BMO space $BMO(L^\alpha)$ with a  bound $\lesssim (1+|s|)^{|\frac32|}\exp(|\frac {|\pi s|}2|)$ (see (\ref{Lis}),(\ref{Lisp})). This   complements  Cowling's $L^p$-estimate (see \cite[Corollary 1]{Co83})  and fixes a mistake in  \cite{JM12} (see the Remark at the end of Section 3).

The related topics and  estimates on semigroup generators have been studied with  geomtric/metric assumptions on  the underlying measure space. This article is from a functional analysis point of view and tries to obtain a general result  by abstract arguments. Cowling and Hieber/Pr\"uss's method for their $H^\infty$-calculus results on $L^p$ is  based on the transference techniques of Coifman and Weiss, which does not work for non-UMD Banach spaces, such as BMO. Our method is to consider the fractional power of the generator to take   advantage of the quasi-monotone property (\ref{quasi}). Our argument works well for the noncommutative case, that is for $L$ that generates a semigroup of completely positive contractions on a semifinite von Neumann algebra. 

We analyze a few examples to illustrate our results  and demonstrate their applications to  Fourier multipliers on non-classical $L^p$ spaces  at the end of the article. We use $c$ for an absolute constant which may differ from line to line.

\section{The complete monotonicity of a difference of exponential power functions}

A nonnegative $C^\infty$-function $f(t)$ on $(0,\infty)$ is {\it completely monotone} if $$(-1)^k\partial_t^kf(t)\geq0$$ for all $t$. Easy examples are 
$f(t)=e^{-\la t}$ for any $\la>0$. It is well-known that completely monontonicity is preserved by addition, multiplication, and taking pointwise limits. So the Laplace transform  of a positive Borel measure on $[0,\infty)$, which is an average of  $e^{-\la t}$ in $\la$, is completely monotone.
The Hausdorff-Bernstein-Widder Theorem says that the reverse is also true; namely that a function is   completely monotone if and only if it is the Laplace transform  of a positive Borel measure on $[0,\infty)$. 
In particular, $g_s(t)=e^{-st^\alpha}$ is completely monotone and is the Laplace transform of a positive integrable $C^\infty$ function $\phi_{s,\alpha}$ on $(0,\infty)$ for all $s>0,0<\alpha<1$.
\begin{eqnarray}\label{est}
e^{-st^\alpha}=\int_0^\infty e^{-\la t}\phi_{s,\alpha}(\la)d\la=\int_0^\infty e^{-s^\frac1\alpha\la t}\phi_{1,\alpha}(\la)d\la.
\end{eqnarray}
The function $\phi_{s,\alpha}$ is  uniquely  determined   by the inverse Laplace  transform 
\begin{eqnarray}\label{rest}
\phi_{s,\alpha}(\la)=s^{-\frac1\alpha }\phi_{1,\alpha}(s^{-\frac1\alpha}\la)={\mathcal L}^{-1}(e^{-sz^\alpha})(\la)=\frac1{2\pi i}\int_{\sigma-i\infty}^{\sigma+i\infty} e^{z\la}e^{-s z^\alpha}dz, 
\end{eqnarray}
for $\sigma>0,\la>0$. The derivative $\partial_s\phi_{s,\alpha}$ is again an integrable function (see e.g. {\cite[page 263]{Y80}), and 
\begin{eqnarray}\label{dest}
-t^{\alpha}e^{-st^\alpha}=\int_0^\infty e^{-\la t}\partial_s\phi_{s,\alpha}(\la)d\la.
\end{eqnarray}
The properties of $\phi_{s,\alpha}$ are important in the study of the fractional powers of   semigroup generators. 

The goal of this section is to prove  a few pointwise inequalities  for $\phi_{s,\alpha}$, which  will  be used in the next section. For that purpose, we first prove the complete monotonicity of several variants of $e^{-st^\alpha}$.

For $k,n\in {\Bbb N}, 1\leq k\leq n$, let $a_k^{(n)}$ be the real coefficients in the expansion  
\[
\frac{d^n}{dt^n} e^{-t^\alpha} = (-1)^n
\sum_{k=1}^n a_k^{(n)} t^{-n+k\alpha} e^{-t^\alpha}.
\]
 It is easy to see that
\[
\frac{d^n}{dt^n} e^{-ct^\alpha} = (-1)^n
\sum_{k=1}^n c^k a_k^{(n)} t^{-n+k\alpha} e^{-c t^\alpha}.
\]
 
\noindent{\bf Convention:} We define $a_k^{(n)} = 0$ if $k > n$ or $k\leq 0$. 

\medskip

The proof of the following lemma is simple and elementary. We leave it for the reader to verify. 

\begin{lemma}
The $a_k^{(n)}$'s satisfy the relation 
\begin{eqnarray}\label{basic}
a_{k}^{(n+1)} = (n-k\alpha)  a_k^{(n)} + \alpha a_{k-1}^{(n)}
\end{eqnarray}
for all $k\in {\Bbb Z}, n\in {\Bbb N}$.  
\end{lemma}

\begin{lemma}\label{thm:akn-subexp}
Let $K_i,i=1,2$ be the first  integer $m$ such that $\frac m{m+i} \geq \alpha$.  
Then, for all $   j \in {\Bbb Z}, n\in {\Bbb N}$, we have
 
\begin{eqnarray}
a_{k+j}^{(n)} - (j+1) a_{k+j+1}^{(n)}&\geq&0 \ \ {\rm if}\ k\geq K_1 \label{key}\\
(j+1) ( a_{k+j+1}^{(n)}-(j+2)a_{k+j+2}^{(n)})&\leq& a_{k+j}^{(n)}-(j+1)a_{k+j+1}^{(n)} \ \ {\rm if}\ k\geq K_2 \label{key2}
\end{eqnarray} 

\end{lemma}
\begin{proof} We only need to prove the case $j\geq0$. Let $D$ be the right   derivative for discrete functions:
$Df=f(j+1)-f(j)$. It is easy to see that  the product rule holds  $D(jf)(j)=jD_jf(j)+f(j+1)$. Fix $k\in {\Bbb Z}$. Let \begin{eqnarray}
\label{fn}
f_n(j)=a_{k+j}^{(n)}j!
\end{eqnarray} for $j \geq0$, where we use the convention that $0!=1$.
By  (\ref{basic}), we have
$$f_{n+1}(j)=(n-(k+j)\alpha)f_n(j)+\alpha jf_n(j-1),$$
for all $j\geq 1$ and $f_{n+1}(0)=(n-k\alpha)f_n(0)+\alpha a_{k-1}^{(n)}$.
Taking the discrete derivative on  both sides, we get
\begin{eqnarray}
 Df_{n+1}(j) &=&(n-(k+j)\alpha)Df_n(j)-\alpha f_n(j+1)+\alpha jDf_n(j-1)+\alpha f_n(j)\nonumber \\
 &=&(n-(k+j+1)\alpha)Df_n(j) +\alpha jDf_n(j-1). \label{newD}
  \end{eqnarray}
  for $j\geq1$ and $ Df_{n+1}(0) =(n-(k+1)\alpha)Df_n(0) -\alpha a_{k-1}^{(n)}.$ 
  By induction, we get
  \begin{eqnarray}
 D^if_{n+1}(j) =(n-(k+j+i)\alpha)D^if_n(j) +\alpha jD^if_n(j-1).\label{newD2}
  \end{eqnarray}
  for all $i\geq1,j\geq1$ and $ D^if_{n+1}(0) =(n-(k+i)\alpha)D^if_n(0) +(-1)^i\alpha a_{k-1}^{(n)}.$ 
  
  Let $k=K_1$ in (\ref{fn}). Note that  the condition $Df_n(j)\leq 0$  trivially holds for $n\leq K_1+j$  because $a_{i}^{(j)}=0$ for $i>j$. In particular, $Df_n(j)\leq 0$  for all $j\geq0, n= K_1 $. We apply induction on $n$. Assume $Df_n(j)\leq 0$ holds for all $j\geq0$.  The equality (\ref{newD}) implies that $Df_{n+1}(j)\leq 0$ for all  $j\geq0$ satisfying $n\geq (K_1+j+1)\alpha$, which holds if  $ n+1\geq K_1+j+1$ since $\frac n{n+1}\geq \alpha$.  On the other hand, if $n+1\leq K_1+j$ we have $Df_{n+1}(j)\leq 0$ trivially. So $Df_{n+1}(j)\leq 0$ for all $j\geq0$. Therefore,  $Df_{n}(j)\leq 0$ and equivalently  (\ref{key}) holds for all $n\in {\Bbb N}, j\geq0.$

The argument for (\ref{key2}) is similar. Let $k=K_2$ in (\ref{fn}). Note that $D^2f_n(j)\geq 0$ is equivalent to (\ref{key2}) for $j\geq0$, which  trivially holds for   $n\leq K_2+j$ since $K_2\geq K_1$ and $a_{K_2+j}^{(n)}-(j+1)a_{K_2+j+1}^{(n)}\geq0$. In particular,   (\ref{key2}) holds for  $n= K_2, j\geq0$.  Assume   that (\ref{key2}) holds for  $n= m,  j\geq0  $.    We consider the case $ n=m+1$. If  $n=m+1\leq K_2+j$, (\ref{key2}) holds   trivially.  Otherwise, $m+1\geq K_2+j+1 $ and by applying (\ref{newD2}) we see that $D^2f_{n+1}\geq 0$. By induction, (\ref{key2}) holds for all $n\in {\Bbb N},j\geq0.$
  \end{proof}

\medskip
\noindent {\bf  Remark.} The argument of the previous lemma shows that $(-1)^iD^if_n(j)\geq 0$ for all $n\in{\Bbb N},j\geq0$ if we choose $k $ so that $\frac k{k+i}\leq \alpha$.

\medskip


For a fixed $K\geq K_1$, let  
\begin{eqnarray}\label{Fn}
F_n(x)=x^{-K}\sum_{j=1}^{n}a_j^{(n)}x^j=\sum_{j=-\infty}^\infty a_{K+j}^{(n)} x^j.
\end{eqnarray}
 and for a fixed $K\geq K_2$, let  
\begin{eqnarray}\label{Gn}G_n(x)=x^{-K}\sum_{j=1}^{n+1}(a_{j-1}^{(n)}-(j-K)a_{j}^{(n)})x^{j-1}=\sum_{j=-\infty}^\infty (a_{K+j-1}^{(n)}-ja_{K+j}^{(n)}) x^{j-1}.
\end{eqnarray}
\begin{lemma}\label{thm:subexp-doublegrowth}
Let $f(x)=F_n(x)$, or $G_n(x)$ for the given suitable $K$. We have $(f(x)e^{-x})'\leq 0$ and $f(x+rx) \leq e^{rx} f(x)$ 
for all $r,x > 0$. 
\end{lemma}

\begin{proof} It is easy to see that $f(x)-f'(x)\geq0$ for $x>0$ by Lemma \ref{thm:akn-subexp}. So
$(f(x)e^{-x})'=(f'-f)e^{-x}\leq0$ and hence   $f(x+rx)  \leq e^{rx} f(x)$ for $r>0$.
\end{proof}

We now come to the main result of this section. 
\begin{theorem}\label{main1}
Let $0 < \alpha$, $c < 1$, and $s\geq 0$ be fixed.
Then  \begin{itemize}
\item[(i)] $e^{-cst^{\alpha}} - c^{K_1}e^{-s t^{\alpha}}$ is completely monotone in $t$.
\item[(ii)] $K_1e^{-st^{\alpha}} +st^\alpha e^{-st^{\alpha}} $ is completely monotone in $t$. 

\item[(iii)]  $\frac1{c^{K_2}(1-c) } e^{-cst^{\alpha}} -st^\alpha e^{-st^{\alpha}} $ is completely monotone in $t$. 
\item[(iv)]  $(\max\{\frac {jK_1}{c^{K_1}}, \frac j{c^{K_2}(1-c)}\})^j e^{-cst^{\alpha}} \pm s^jt^{j\alpha} e^{-st^{\alpha}} $ are completely monotone in $t$ for any $j\in {\Bbb N}$.
\end{itemize}
\end{theorem}

\begin{proof} By dilation, we may assume $s=1$.
We prove (i) first.  Let $x=t^\alpha$ and $F_n$ be as in \ref{Fn},
\[
\frac{d^n}{dt^n} e^{-t^\alpha} = (-1)^n t^{-n} 
\sum_{k=1}^n a_k^{(n)} x^k e^{-x}=(-1)^nt^{-n+K\alpha}e^{-x}F_n(x)
\]
and
\begin{eqnarray}\label{cx}
\frac{d^n}{dt^n} e^{-ct^\alpha} = (-1)^n t^{-n} 
\sum_{k=1}^n c^k a_k^{(n)} x^k e^{-x}e^{-rx}=(-1)^nt^{-n+K\alpha}c^{K}e^{-cx}F_n(cx) .
\end{eqnarray}
Applying Lemma \ref{thm:akn-subexp} and Lemma \ref{thm:subexp-doublegrowth} to $F_n$ gives us    
$$\frac{\frac{d^n}{dt^n} e^{-ct^\alpha} }{\frac{d^n}{dt^n} e^{-t^\alpha}} \geq c^{K},$$
for any $K\geq K_1$.
This implies (i) since $ e^{-t^\alpha}$ is completely monotone for any $0<\alpha\leq1$.

We now prove (ii). Let $g(s,t)=e^{-st^\alpha} {s^{-{K_1}}}$. Then 
$-\partial_s g(s,t)$, is the limit of the family of functions $$ \frac 1{s^{{K_1}+1}(c-1)}(e^{-st^{\alpha}} - c^{-{K_1}} e^{-cst^{\alpha}})$$ as $c\rightarrow1$, which are  completely monotone in $t$ by (i). So $${K_1}e^{-st^{\alpha}} +st^\alpha e^{-st^{\alpha}} =-s^{{K_1}+1}\partial_s g(s,t)$$ is completely monotone in $t$.

For (iii), we denote by $f^{(n)}(t)=\partial_t^nf(t)$ and, for $K\geq K_2\geq K_1$, write
\begin{eqnarray}
(t^\alpha e^{-t^{\alpha}})^{(n)}+  {K} (e^{-t^{\alpha}})^{(n)}&=&-\frac1\alpha[t (e^{-t^{\alpha}})']^{(n)}+ {K} (e^{-t^{\alpha}})^{(n)}\nonumber\\
&=&-\frac1\alpha[t (e^{-t^{\alpha}})^{(n+1)}+n(e^{-t^{\alpha}})^{(n)}]+  {K} (e^{-t^{\alpha}})^{(n)}\nonumber\\
&=&\frac{(-1)^{n}t^{-n}}\alpha\left[\sum_{k= 1}^\infty(a_k^{(n+1)}-(n-K\alpha)a_k^{(n)})t^{k\alpha}e^{-t^\alpha}\right]\nonumber\\
&=& {(-1)^{n}t^{-n}} \left[\sum_{k=1}^\infty(a_{k-1}^{(n)}-(k-K)a_k^{(n)})t^{k\alpha}e^{-t^\alpha}\right]\nonumber\\
&=& {(-1)^{n}t^{-n+K\alpha}} \left[\sum_{k=-\infty}^\infty(a_{K+k-1}^{(n)}-ka_{K+k}^{(n)})t^{k\alpha}e^{-t^\alpha}\right]\nonumber\\
&=& {(-1)^{n}t^{-n+{K}\alpha}}  x e^{-x} G_n(x)\label{tout}
\end{eqnarray}
 with $x=t^\alpha$ and $G_n(x)$ defined as in   \ref{Gn}, which depends on $K$. Lemma  \ref{thm:subexp-doublegrowth} says that $G_n(x)e^{-x}$ deceases in $x$ if $K\geq K_2$ and note that $G_n(x)e^{-x} =-(F_n(x)e^{-x})'\geq0$. We have 
\begin{eqnarray*}
xG_n(x)e^{-x}&\leq&\frac1{(1-c)}\int_{cx}^{x}G_n(s)e^{-s}ds\\
 &=&\frac1{(1-c)}\int_{cx}^{x}-(F_n(s)e^{-s})'ds\\
 &\leq&\frac1{(1-c)} F_n(cx)e^{-cx},
\end{eqnarray*}
for $0<c<1$.  
Combing this inequality with (\ref{cx}) and (\ref{tout}) we get
$$\frac{(-1)^n\frac{d^n}{dt^n} (t^\alpha e^{- t^\alpha} +  {K_2} e^{-t^{\alpha}})}{(-1)^n\frac{d^n}{dt^n} e^{-ct^\alpha}} \leq  \frac1{c^{K_2}(1-c) }.$$
This proves (iii) since $e^{-ct^\alpha}$ and $e^{-t^\alpha}$ are completely monotone.

For (iv), let $f(t)=\max\{\frac {K_1}{c^{K_1}}, \frac1{c^{K_2}(1-c)}\} e^{-cst^{\alpha}}, g(t)=st^{\alpha}e^{-st^{\alpha}}$. By (i), (ii) and (iii) we have 
that both $f+g,f-g$ are completely monotone in $t$. Recall that complete monotonicity is preserved by multiplication. Note that  
\begin{eqnarray*}
f^{j+1}+g^{j+1}=\frac12[(f^j-g^j)(f-g)+(f^j+g^j)(f+g)] \\
f^{j+1}-g^{j+1}=\frac12[(f^j-g^j)(f+g)+(f^j+g^j)(f-g)].
\end{eqnarray*}
We get, by induction, that $(\max\{\frac {K_1}{c^{K_1}}, \frac1{c^{K_2}(1-c)\alpha}\})^j e^{-jcst^{\alpha}} -s^jt^{j\alpha} e^{-jst^{\alpha}} $ is completely monotone for any $s>0$, which implies (iv).
\end{proof}

We will apply Theorem \ref{main1} to  pointwise estimates of $\phi_{s,\alpha}(\la)$. Let us first list a few basic properties of $\phi_{s,\alpha}.$
\begin{lemma} For any $s>0, 0<\alpha,\beta<1$, we have
\begin{eqnarray}
\phi_{s,\frac12}(\la)&=&\frac 1{2\sqrt{\pi }}se^{-\frac{s^2}{4\la}}\la^{-\frac
32} . \label{1/2}\\
\phi_{1,\alpha\beta}(\la)&=&\int_0^\infty \phi_{s,\alpha}(\la)\phi_{1,\beta}(s)ds. \label{cvlt}\\
\phi_{s,\alpha}(\la)&=&s^{-\frac1\alpha}\phi_{1,\alpha}(s^{-\frac1\alpha}\la),\label{sla}\\
-\alpha s\partial_s\phi_{s,\alpha}( \la)&=&\phi_{s,\alpha}( \la)+\la\partial_\la \phi_{s,\alpha}( \la). \label{dsla}
\end{eqnarray}
\end{lemma}

\begin{proof} (\ref{1/2}) is well-known (see e.g. \cite{Y80}, page 268). (\ref{cvlt}), (\ref{sla}) can be easily seen from (\ref{est}) and (\ref{rest}).  (\ref{sla}) implies (\ref{dsla}).
\end{proof}

\begin{corollary}\label{fs}
For all  $\la,s>0, 0<c<1,j\in {\Bbb N}$, we have 
\begin{eqnarray}
\label{csf1}c^{K_1}\phi_{s,\alpha}(\la)&\leq&  \phi_{cs,\alpha}(\la)\\
0&\leq&\partial_\la(\la^{1+ \alpha K_1  }\phi_{s,\alpha}(\la)),\label{csfla}\\
|s^j\partial^j_s\phi_{s,\alpha}(\la)|&\leq& \left(\max\left\{\frac {jK_1}{c^{K_1}}, \frac j{c^{K_2}(1-c)\alpha}\right\}\right)^j \phi_{cs,\alpha},\\
 |s\partial_s\phi_{s,\alpha}(\la)|&\leq& \left(\frac {10}{1-\alpha}\right)\phi_{\alpha s,\alpha}(\la)\label{cdsf},\\
  |s^j\partial^j_s\phi_{s,\alpha}(\la)|&\leq& \left(\frac {10j}{1-\alpha}\right)^j\phi_{\alpha s,\alpha}(\la)\label{cdsf3}.
\end{eqnarray}
\end{corollary}

\begin{proof} These are direct consequences of Theorem \ref{main1}, the identity (\ref{est}), and the Hausdorff-Bernstein-Widder Theorem because $K_i\leq \frac i{1-\alpha}$, except that (\ref{csfla}) requires a little more calculation. To prove (\ref{csfla}), note that (\ref{dest}) and Theorem \ref{main1} (ii) imply that 
$$\partial_s\frac {\phi_{s,\alpha}(\la)}{s^{K_1}}=-s^{-K_1-1}(K_1\phi_{s,\alpha}(\la)-s\partial_s\phi_{s,\alpha}(\la))\leq 0.$$
Since $\phi_{s,\alpha}(\la)=s^{-\frac1\alpha}\phi_{1,\alpha}(s^{-\frac1\alpha}\la)$, we get
$$
-\left(\frac1{\alpha}+K_1\right)s^{-\frac1{\alpha}-K_1-1}\phi_{1,\alpha}(s^{-\frac1\alpha}\la)-\frac1\alpha s^{-\frac1\alpha -1}\la s^{-\frac1\alpha-K_1 }(\partial_\la\phi_{1,\alpha})(s^{-\frac1\alpha}\la)\leq 0.
$$
That is 
$$
(1+K_1\alpha) \phi_{1,\alpha}(s^{-\frac1\alpha}\la)+  \la s^{-\frac1 \alpha }(\partial_\la\phi_{1,\alpha})(s^{-\frac1\alpha}\la)\geq0 .
$$
Therefore
$$
(1+K_1\alpha) \phi_{s,\alpha}( \la)+  \la  \partial_\la\phi_{s,\alpha}( \la)\geq0,
$$
since $\partial_\la \phi_{s,\alpha}(\la)=s^{-\frac2\alpha}\partial_\la \phi_{s,\alpha}(s^{-\frac1\alpha}\la)$.
This is (\ref{csfla}). 
\end{proof}

\begin{lemma}
For any $s>0, 0<\beta<\alpha<1$, we have that
\begin{eqnarray}
\int_0^\infty  \left|\ln  \left(s^{-\frac1\alpha}u\right)\right| \phi_{s,\alpha} (u)du<\frac c{\beta}. \label{lnuv}\\
\int_0^\infty\int_0^\infty \left|\ln  \left( \frac  uv\right)\right|\phi_{s,\alpha} (u)\phi_{s,\alpha}(v)dudv<\frac c{\beta^2}. \label{lnuv2}
\end{eqnarray}
\end{lemma}

\begin{proof} Since $\phi_{s,\alpha}(u)=s^{-\frac1\alpha}\phi_{1,\alpha}(s^{-\frac1\alpha}u)$, the left hand side of (\ref{lnuv}) is independent of $s$. We only need to prove the case $s=1$. For $\alpha=\frac12$, we can verify directly from (\ref{1/2}) that  (\ref{lnuv}) holds. Denote by $u(\alpha)$ the left hand side of (\ref{lnuv}). We then get $u(\frac1{2 })<\infty$. Using (\ref{cvlt}), we  get $u(\frac1{2^n})<\infty$. 
Now, for $\alpha>\frac1{2^n}$, we use (\ref{cvlt}) again and get
 \begin{eqnarray*}
\phi_{1,\frac1{2^n} }(\la)&=&\int_0^\infty \phi_{s,\alpha}(\la)\phi_{1,\frac1{\alpha2^n} }(s)ds\\
&\geq& \int_{0}^1 \phi_{s,\alpha}(\la)\phi_{1,\frac1{\alpha2^n} }(s)ds\\
(by \ (\ref{csf1}))&\geq&\phi_{1,\alpha}(\la) \int_{0}^1 s^{K_1(\alpha)}\phi_{1,\frac1{\alpha2^n} }(s)ds\\
&\geq&   c_\alpha \phi_{1,\alpha}(\la).
\end{eqnarray*}
We conclude that $u(\alpha)<\infty$ for all $0<\alpha<1$. Since $\phi_{1,\alpha}(\la)$ is continuous as a function in $\alpha$ and this continuity is uniform  for $\la\in [\delta,N]$ for any $0<\delta<N<\infty$, one can easily see that $u(\alpha)$ is continuous in $\alpha$ for $\alpha\in (0,1)$.
We conclude that $u(\alpha)$ is bounded on $[\frac1{2^n},\frac12]$ for any $n\in {\Bbb N}$.
Note  that (\ref{cvlt}) also implies that  
 \begin{eqnarray}
&&\int_0^\infty\phi_{1,\alpha\beta }(\la)|\ln \la|d\la\nonumber\\&=&\int_0^\infty \int_0^\infty \phi_{s,\alpha}(\la)|\ln\la|d\la\phi_{1,\beta }(s)ds\nonumber\\
&=&  \int_0^\infty \int_0^\infty \phi_{1,\alpha}(v)|\ln(s^{\frac1\alpha}v)|dv\phi_{1,\beta }(s)ds\nonumber\\
&\geq&\pm  \int_0^\infty \int_0^\infty \phi_{1,\alpha}(v)({\frac1\alpha}|\ln s |-|\ln v|) dv\phi_{1,\beta }(s)ds\label{lphabet}\\
\bigg(&\leq&  \int_0^\infty \int_0^\infty \phi_{1,\alpha}(v)({\frac1\alpha}|\ln s|+|\ln v|) dv\phi_{1,\beta }(s)ds\bigg)\label{phabe}
\end{eqnarray}
Our change in the order of integration is justified because all the terms are positive. Note $\int_0^\infty \phi_{t,\alpha}(s)ds=1$ for any $t,\alpha$. (\ref{lphabet}) and (\ref{phabe}) imply that
\begin{eqnarray}
 |u(\alpha)-\frac1\alpha u(\beta)|\leq u(\alpha\beta ) \leq u(\alpha)+\frac1\alpha u(\beta) \label{alphabeta}
\end{eqnarray}
 We then obtain (\ref{lnuv}). (\ref{lnuv2}) follows from (\ref{lnuv}).
\end{proof}

 \noindent {\bf Remark (Bell Polynomials).}
We define the complete Bell polynomial $B_n(x_1,\ldots,x_n)$ by its generating function
\[
\exp\left(\sum_{j=1}^{\infty}x_j\frac{u^j}{j!}\right)=\sum_{n=0}^{\infty}B_n(x_1,\ldots,x_n)\frac{u^n}{n!}
\]
From this, we get the formula
\[
B_n(x_1,\ldots,x_n)=\frac{d^n}{du^n}\exp\left(\sum_{j=1}^{\infty}x_j\frac{u^j}{j!}\right)\bigg|_{u=0}
\]
Now, for $s>0$, let
\begin{equation}\label{xjs}
x_j=-s\frac{d^j}{dt^j}t^{\alpha}=-s(\alpha)_jt^{\alpha-j},
\end{equation}
where $(\alpha)_j$ denotes the falling factorial.  Then
\[
\sum_{j=1}^{\infty}x_j\frac{u^j}{j!}=-st^{\alpha}\sum_{j=1}^{\infty}\frac{(\alpha)_j}{j!}\left(\frac{u}{t}\right)^j=st^{\alpha}-st^{\alpha}\left(1+\frac{u}{t}\right)^{\alpha}=st^{\alpha}-s(t+u)^{\alpha}
\]
Applying Theorem \ref{main1} part (i), we see that for all $n\in\bbN$, $c\in(0,1)$, and $t>0$ it holds that
\[
\frac{\frac{d^n}{du^n}e^{-sc(t+u)^{\alpha}}\bigg|_{u=0}}{\frac{d^n}{du^n}e^{-s(t+u)^{\alpha}}\bigg|_{u=0}}\geq c^{K_1},
\]
where ${K_1}$ is as in Lemma \ref{thm:akn-subexp}.  We can rewrite this inequality as
\[
e^{(1-c)st^{\alpha}}\frac{\frac{d^n}{du^n}e^{sct^{\alpha}-sc(t+u)^{\alpha}}\bigg|_{u=0}}{\frac{d^n}{du^n}e^{st^{\alpha}-s(t+u)^{\alpha}}\bigg|_{u=0}}\geq c^{K_1}.
\]
We conclude that if we define $x_j$ by \eqref{xjs}, then
\begin{equation}\label{bell}
e^{(1-c)st^{\alpha}}\frac{B_n(cx_1,\ldots,cx_n)}{B_n(x_1,\ldots,x_n)}\geq c^{K_1}
\end{equation}
for all $n\in\bbN$, $c\in(0,1)$, and $t>0$.  All of these calculations are easily reversible, and we conclude that \eqref{bell} is actually equivalent to part (i) of Theorem \ref{main1}.

\section{ Positive semigroups and BMO}

Let $(M ,\sigma , \mu )$ be a sigma-finite measure space. Let $L^1(M)$ be the space of all complex valued integrable  functions  and $L^\infty(M)$ be the space of all complex valued measurable and essentially bounded  functions on $M$. Denote by $f^*$  the pointwise complex conjugate of a function $f$ on $M$ and by $\langle f,g\rangle$ the duality bracket $\int fg^*$.

\begin{definition}\label{standard}
A map $T$ from $L^\infty(M)$ to $L^\infty(M)$ is called {\it positive } if $Tf\geq 0$ for $f\geq 0$.   If $T$ is positive on $L^\infty(M)$, then  $T\otimes id$ is positive  on matrix valued function spaces $ L^\infty(M)\otimes M_n$ for all $n\in {\Bbb N}$, i.e. $T$ is {\it completely positive}. 
\end{definition}

A positive map $T$ commutes with complex conjugation, i.e. $T(f^*)=T(f)^*$. For two positive maps $S,T$, we will write $S\geq T$ if $S-T$ is positive.

We will need the following Kadison-Schwarz inequality for  
completely positive maps $T$,
\begin{eqnarray}\label{cp}
|T(f)|^2\leq \|T(1)\|_{L^\infty}T(|f|^2),\qquad\quad f\in L^\infty(M).
\end{eqnarray}

\subsection{Postive semigroups}
We will consider a semigroup  $(T_t)_{t\geq0}$ of positive, weak*-continuous contractions on $L^\infty$ with the weak* continuity at $t=0+$. That is a family of positive, weak*-continuous contractions $T_t, t\geq0$ on $L^\infty$ such that $T_sT_t=T_{s+t}$, $T_0=id$ and  
 $\langle T_t(f), g \rangle \rightarrow \langle f,g\rangle$ as $t\rightarrow0+$ for any $f\in L^\infty,g\in L^1.$ 
 
 Such a semigroup $(T_y)$ always admits an
infinitesimal negative generator $L=\lim_{y\rightarrow 0}\frac {id-T_y}{y}$ which has a weak*-dense domain $D(L)\subset L^\infty$.  
We will write $T_y=e^{-yL}$. These definitions and facts extend to the noncommutative setting. Namely, given a semifinite von Neumann algebra $\M$ and a normal semifinite  faithful trace  $\tau$, we let $L^\infty(\M)=\M$ and $L^1(\M)$ be the completion of $\{f\in \M: \|f\|_{L^1}=\tau|f|<\infty\}$.  Here $ |g|=(g^*g)^\frac12$ and $g^*$ denotes the adjoint operators of $g$ and we set $\langle f,g\rangle=\tau (fg^*)$. We say a map $T$ on $\M$ is   completely positive if $(T\otimes id)(f)\geq 0$ for any $f\geq0, f\in \M\otimes M_n$. We say $f_\la$  weak*  converges to $f$  if $\lim_\la \langle f_\la ,g\rangle=\langle f,g\rangle$ for all $g\in L^1(\M)$ (see \cite{JMX06} for details).


The so-called subordinated semigroups $T_{y,\alpha} =e^{-y L^\alpha}, 0<\alpha<1$ are defined as 
\begin{eqnarray}
T_{t,\alpha}f=\int_0^\infty T_uf\phi_{t,\alpha}(u)du=\int_0^\infty T_{t^\frac1\alpha u}f\phi_{1,\alpha}(u)du,  \label{idpy}
\end{eqnarray}
with $\phi_{t,\alpha}$ given in Section 1. 
 The generator  $L^\alpha$ is given by
\begin{eqnarray}\label{B60}
L^\alpha(f)=\Gamma(-\alpha)^{-1}\int_0^\infty (T_t-id)(f)t^{-1-\alpha}dt,
\end{eqnarray}
for $f\in D(L)$.
There are other (equivalent) formulations for $L^\alpha$. The   formula (\ref{B60}) is due to Balakrishnan (see \cite{B60} and \cite[page 260]{Y80}). For  $T_t=e^{-tz}id$ with $Re(z)\geq0$, $L^\alpha=z^\alpha$ with a chosen   principal value  so that $Re({z}^\alpha)\geq0$. 

$(T_{y,\alpha})$ is again a semigroup of positive weak*-continuous contractions. The semigroup has an analytic extension and has the well-known norm estimate that 
\begin{eqnarray}\label{analy}
\sup_{y>0}\|y^k\partial^k_yT_{y,\alpha}\|<c_k. \end{eqnarray}
What we wish is a pointwise estimate.


Note that (\ref{idpy}) implies
\begin{eqnarray}
\frac{T_{y,\frac12}}y(f)\leq \frac{T_{t,\frac12}}t(f) \label{sbd} \qquad {\rm and}\qquad
|y^k\partial_y T_{y,\frac12}f|\leq c_{k,t} T_{t,\frac12}f,
\end{eqnarray}
for any $0\leq t\leq y$ and $f\geq0$
because of the positivity of $T_u$  and the precise formulation of $\phi_{y,\frac12}$. 

 Corollary \ref{fs} and the identity (\ref{idpy}) actually imply the following corollary.

\begin{corollary}\label{Ts}
For all  $f\geq0,s>0, 0<c,\alpha<1$, and $j\in{\Bbb N}$, we have 
\begin{eqnarray}
c^{K_1}T_{s,\alpha}f&\leq&  T_{cs,\alpha}f \label{csf}\\
 |s^j\partial^j_sT_{s,\alpha}(f)|&\leq& (\frac {10j}{1-\alpha})^jT_{\alpha s,\alpha}(f)\label{cdsf2}.
\end{eqnarray}
\end{corollary}
\medskip
\noindent{\bf Remark.} When $\alpha=1$, a similar estimate to Corollary \ref{Ts} may hold for some special semigroups. For example, the heat semigroups generated by the Laplacian operator on ${\Bbb R}^n$ has a similar estimate   with $ c>1$. But one can not hope this in general since (\ref{cdsf2}) is   stronger   then the analyticity on $L^\infty$.

\subsection{$\Gamma^2$ criterion} P. A Meyer's gradient form ${\Gamma}$ (also called ``Carr\'e du Champ") associated with $T_t$ is defined as,
 \begin{eqnarray}\label{Gamma}
  2{\Gamma}_L(f,g) =- L(f^{*}g)+(L(f^{*})g)+f^{*}(L(g)),
  \end{eqnarray}
for $f,g$ with $f^*,g, f^*g\in D(L)$. It is easy to verify that for $L=-\triangle=-\frac{\partial ^2}{\partial^2 x}$, $\Gamma_L(f,g)= \nabla f^* \cdot \nabla g $. 

\medskip
\noindent{\bf Convention.}  
We will write $\Gamma (f)$ for $\Gamma_L (f,f)$.
\medskip

It is well known that the completely positivity of  the operators $T_t$ implies that $\Gamma(f,g)$ is a completely positive bilinear form. We then have the Cauchy-Schwartz inequality
\begin{eqnarray}
 \Gamma \left(\int_0^\infty a_sd\mu(s),\int_0^\infty a_sd\mu(s)\right) 
\leq  
\int_0^\infty d|\mu|(s)\int_0^\infty \Gamma (a_s,a_s)d|\mu |(s)  \label{cauchy}
\end{eqnarray}

Bakry-\'Emery's $\Gamma^2$ criterion plays an important role in this article. We use an equivalent definition.
\begin{definition} A   semigroup of positive operator $(T_t)_t$ satisfies the $\Gamma^2\geq0$ criterion if $\Phi(s)=T_{s-u}|T_uf|^2, s>u$ is (midpoint) convex in $u$, i.e.
 \begin{eqnarray}\label{Gamma2}
  T_t|T_uf|^2-|T_tT_uf|^2\leq T_u(T_{t}|f|^2-|T_{t}f|^2)
  \end{eqnarray}
for all $t,u>0$ and $f\in L^\infty$. 
\end{definition}

For $L$ equal to the Laplace-Beltrami operator on a complete manifold, the $\Gamma^2\geq0$ criterion holds if the manifold has   nonnegative   Ricci curvature everywhere. 
The ``$\Gamma^2$" criterion  is satisfied by a large class of semigroups including the heat, Ornstein-Uhlenbeck, Laguerre, and Jacobi semigroups (see \cite{Ba96}), and also by the  semigroups of completely positive contractions   on group von Neumann algebras.
We refer the reader to \cite{BBG12} and references therein for the so-called curvature-dimension criterion  which is more general than  the ``$\Gamma^2 $" criterion.

D. Bakry usually assumes
that there exists a $^*$-algebra ${\mathcal A}$ which is weak$^*$ dense in
$L^\infty(M)$ such that $T_s({\mathcal A})\subset {\mathcal A}\subset D(L)$. 
This is not needed in this article because we will only use the form $T_{t,\alpha}{\Gamma}_{L^\beta}(T_{s,\alpha}f,T_{s,\alpha}g), 0<\alpha<1,\alpha\leq \beta\leq1$ which is well defined as
\begin{eqnarray}\label{Gamma1}
-L^\beta T_{t,\alpha}[(T_{s,\alpha}f^{*})(T_{s,\alpha}g)]+T_{t,\alpha}[(T_{s,\alpha}f^{*})(L^\beta T_{s,\alpha}g)]+T_{t,\alpha}[(L^\beta T_{s,\alpha}f^{*})(T_{s,\alpha}g)]
\end{eqnarray}
  for all $f,g\in L^\infty $ since $T_{s,\alpha}(L^\infty)\subset D (L)\subset D (L^\alpha)$ because of (\ref{idpy}). 

We will need the following Lemma due to P.A. Meyer. We add a short proof for the convenience of the reader.
\begin{lemma}\label{Meyer}
 For any $f\in L^\infty $ such that $T_sf,T_sf^*,T_s|f|^2\in D(L)$ for all $s>0$, we have
 \[ T_s|f|^2-|T_sf|^2 = 2 \int_{0}^s  T_{s-t}\Gamma(T_tf) dt  .\]
 In particular, for $0<\alpha<1$,
 \begin{eqnarray}\label{Meyeralpha}
  T_{s,\alpha}|f|^2-|T_{s,\alpha}f|^2 = 2 \int_{0}^s  T_{s-t,\alpha}\Gamma_{L^\alpha}(T_{t,\alpha}f) dt  
  \end{eqnarray}
 for any $f\in L^\infty$.
\end{lemma}

\begin{proof} For $s$ fixed, let
\[
F_t=T_{s-t}(|T_tf|^2).
\]
Then
\begin{eqnarray}\label{partialF}
\frac{\partial F_t}{\partial t} &=&\frac{\partial T_{s-t}}{%
\partial t}(|T_tf|^2)+T_{s-t}[(\frac{\partial T_t}{\partial t}%
f^{*})f]+T_{s-t}[f^*(\frac{\partial T_t}{\partial t}f)] \nonumber\\
&=&-2T_{s-t}\Gamma (T_tf).
\end{eqnarray}
Therefore
\begin{eqnarray*}
T_s|f|^2-|T_sf|^2 =-F_s+F_0=2\int_0^sT_{s-t}\Gamma (T_tf)dt.
\end{eqnarray*}
Since $T_{s,\alpha}(L^\infty)\subset D(L^\alpha)$ we get (\ref{Meyeralpha}) for all $f\in L^\infty$.
\end{proof}

\noindent{\bf Remark.} Equation (\ref{partialF}) shows that the $\Gamma^2\geq0$ criterion  implies that 
 \begin{eqnarray}\label{Gamma_2}
  T_s\Gamma (T_{v+t} f) \leq T_{v+s} (\Gamma  (T_tf))
  \end{eqnarray}
for all $v,s,t>0$ and $ f\in L^\infty$ such that $T_sf,T_sf^*,T_s|f|^2\in D(L)$ for all $s>0$.

\medskip
The following lemma says that the $\Gamma^2\geq0$ criterion passes to fractional powers, which could be known to some experts. We add a proof as we do not find a  reference.
\begin{lemma}
\label{sub} 
If $T_t=e^{-tL}$ satisfies the $\Gamma^2\geq0$ criterion (\ref{Gamma2}), then $T_{t,\alpha}=e^{-tL^\alpha}$ satisfies   (\ref{Gamma2}) and (\ref{Gamma_2}) for all $f\in L^\infty$ and $0<\alpha<1$. Moreover, 
\begin{eqnarray}
 \Gamma_{L^\alpha} (s^j\partial_s^jT_{s,\alpha}f) 
\leq  \left(\frac{10}{1-\alpha}\right)^jT_{s,\alpha}\Gamma_{L^\alpha} ( f)
   \label{cdsf22}
\end{eqnarray}
 \end{lemma}

 \begin{proof}
 Applying (\ref{B60}), we have that, with $c_\alpha=-(\Gamma(-\alpha))^{-1}>0$, 
 \begin{eqnarray}\label{gammaalpha1}
 \Gamma_{L^\alpha}(f,f)&=&c_\alpha\int_0^\infty (T_t|f|^2-(T_tf^*)f-f^*(T_tf)+|f|^2)t^{-1-\alpha}dt\nonumber\\
 &=&c_\alpha\int_0^\infty (T_t|f|^2-|T_tf|^2+|T_tf-f|^2)t^{-1-\alpha}dt.
 \end{eqnarray} 
 if $f,f^*,|f|^2 \in D(L)$. The integration converges because \begin{eqnarray}\label{1+alpha}
 \|T_t|f|^2-|T_tf|^2\|\leq c\min\{t,1\},
 \end{eqnarray} for $f\in D(L)$. In fact, by the $\Gamma^2\geq0$ criterion (\ref{Gamma2}), we see that
 $$T_t|T_tf|^2-|T_{2t}f|^2\leq \frac12(T_{2t}|f|^2-|T_{2t}f|^2).$$ So
 \begin{eqnarray*}
 \|T_t|f|^2-|T_tf|^2\|^\frac12&\leq& \|T_t|f-T_tf|^2-|T_t(f-T_tf)|^2\|^\frac12+ \|T_t|T_tf|^2-|T_{2t}f|^2\|^\frac12\\
 &\leq&ct+ 2^{-\frac12}\|T_{2t}|f|^2-|T_{2t}f|^2\|^\frac12. 
 \end{eqnarray*}
  Let $u(t)={t^{-\frac{1 }2}\|T_t|f|^2-|T_tf|^2\|^\frac12} $. We get
 $$u(t)\leq ct^{\frac12}+ u(2t).$$
 Since $u(t)$ is uniformly bounded on $[1,\infty)$, we get $u(t)$ is uniformly bounded on $[0,\infty)$ by iteration. This proves (\ref{1+alpha}).
 
 Applying the Cauchy-Schwartz inequality (\ref{cauchy}) and the $\Gamma^2\geq0$ criterion for $T_t$ to (\ref{gammaalpha1}), we get
 \begin{eqnarray}
 \Gamma_{L^\alpha}(T_uf,T_uf)\leq T_u\Gamma_{L^\alpha}(f,f).
 \end{eqnarray}
 Applying the subordination formula that $T_{t,\alpha}=\int_0^\infty T_u\phi_{t,\alpha}(u)du$ and the Cauchy-Schwartz inequality (\ref{cauchy}), we obtain
  \begin{eqnarray}
 \Gamma_{L^\alpha}(T_{t,\alpha} f,T_{t,\alpha}f)\leq T_{t,\alpha}\Gamma_{L^\alpha}(f,f).
 \end{eqnarray}
 One can easily adapt the proof to get
 \begin{eqnarray}\label{gamma2alpha}
T_{u,\alpha} \Gamma_{L^\alpha}(T_{t,\alpha} T_{v,\alpha}g,T_{t,\alpha}T_{v,\alpha}g)\leq T_{u,\alpha}T_{t,\alpha}\Gamma_{L^\alpha}(T_{v,\alpha}g,T_{v,\alpha}g).
 \end{eqnarray}
 for all $g\in L^\infty$ since $T_{v,\alpha}g, T_{u, \alpha}|T_{v,\alpha}g|^2\in D(L)$.
 Applying (\ref{Meyeralpha}), we get (\ref{Gamma_2}) for $T_{t,\alpha}$.
 
 Now, apply (\ref{cauchy}) to $\Gamma_{L^\alpha}$ and $a(s)=T_sf,d\mu(s)=s^j\partial_j\phi_{t,\alpha}(s)ds$; we get (\ref{cdsf22}) from (\ref{idpy}), (\ref{cdsf3}), and (\ref{gamma2alpha}).
 \end{proof}



\subsection{BMO spaces associated with semigroups of operators}

BMO spaces associated with semigroup generators have been intensively studied recently (see \cite{DY05}).
 In this article, we follow the ones studied in \cite{JM12} and \cite{M08} because they are defined in a pure semigroup language. Set
\begin{eqnarray}
\|f\|_{\mathrm{bmo}(L^\alpha) }&=&\sup_{0<t<\infty}\| T_{t,\alpha}|f|^2-|T_{t,\alpha}f|^2\|_{L^\infty}^\frac12,\\
\|f\|_{\mathrm{BMO}(L^\alpha) }&=&\sup_{0<t<\infty}\| T_{t,\alpha}|f-T_{t,\alpha}f|^2\|_{L^\infty}^\frac12, \label{BMOT}
\end{eqnarray}
for $f\in L^\infty , 0<\alpha\leq1$.



We wish to define the space  BMO$( L^\alpha), 0<\alpha\leq1$  so that it is a dual space and $L_0^\infty$ is weak* dense in it, to be consistent with the classical ones (where $L^\infty_0(M)=L^\infty(M)/kerL^\alpha$). In  \cite{JM12} and \cite{M08}, this is done by using a SOT- topology in the corresponding Hilbert C* modulars.   In this article, we prefer to    use the following detour to  avoid introducing the theory of Hilbert C* modulars.   
    Define, for $g\in L^1$,
\begin{eqnarray}\label{H1}
  \|g\|_{H^1(L^\alpha)}=\sup\{|\langle f,g\rangle|: f\in L^\infty, \|f\|_{BMO(L^\alpha)},\|f^*\|_{BMO(L^\alpha)}\leq 1\} .  
 \end{eqnarray}
 Let $H^1(L^\alpha)=\{g\in L^1;  \|g\|_{H^1}<\infty\}$. For a net $f_\la\in L_0^\infty(M)$, we say $f_\la$ converges in the {\it weak* topology} if $\langle f_\la,g\rangle $ converges for any $g\in H^1(L^\alpha)$. Let BMO$( L^\alpha)$ be the abstract closure of $L_0^\infty(M)$ with respect to this weak* topology, that is the linear space of all weak* convergent nets $f_\la \in L_0^\infty(M)$. For a weak* convergent  $f_\la$, let
 $$\|\lim_\la f_\la\|_{BMO(L^\alpha)}=\sup_{\|g\|_{H^1}\leq1}\lim_\la \langle f_\la,g\rangle.$$
 It is easy to see that this  coincides with (\ref{BMOT}) if $\lim_\la f_\la\in L^\infty$.

As an application of Corollary \ref{Ts}, we show that these BMO and bmo norms with different $0<\alpha<1$ are all equivalent if we assume the $\Gamma^2\geq0$ criterion.

\begin{lemma} \label{equibmo} Suppose $L$ generates a  weak* continuous semigroup of positive contractions, we have
\begin{eqnarray}
\|f\|_{BMO (L^\beta)}& \leq& \frac {c\alpha}{\beta}\|f\|_{BMO (L^\alpha)}, \label{alpha<beta}\\
\|f\|_{BMO (L^\beta)} &\leq&\frac{4}{1-\beta}\|f\|_{bmo (L^\beta)},\label{BMO<bmo}
\end{eqnarray}
for any $0<\beta<\alpha\leq1$. Assuming in addition that the semigroup $T_{t}=e^{-tL}$ satisfies the $\Gamma^2\geq0$ criterion (\ref{Gamma_2}), we have that 
\begin{eqnarray}
\|f\|_{BMO(L^\alpha)} \simeq \|f\|_{bmo(L^\alpha)}\simeq \|f\|_{bmo(L^\beta)},\label{BMO=bmo}
\end{eqnarray}
for all $0<\beta,\alpha<1.$ In particular, \begin{eqnarray}
c{ (1-\alpha)^2}\|f\|_{BMO(L^\alpha)}\leq \|f\|_{BMO(\sqrt L)} \leq c \|f\|_{BMO(L^\alpha)}, \label{BMO1/2}
\end{eqnarray}
for all $\frac12<\alpha<1$.
\end{lemma}
\begin{proof}
The argument for (\ref{alpha<beta}) is the same as that for the second inequality of \cite[Theorem 2.6]{JM12}.  We sketch it here. By the Cauchy-Schwartz inequality,
\begin{eqnarray*}  
 T_{t,\beta}|f -T_{t,\beta}f|^2 &=&  \int_0^\infty \phi_{t,\frac \beta\alpha }(u)T_{u,\alpha}  | \int_0^\infty \phi_{t,\frac \beta\alpha }(v)(f -T_{v,\alpha}f)dv|^2 du\\
 &\leq&  \int_0^\infty \int_0^\infty \phi_{t,\frac \beta\alpha }(u) \phi_{t,\frac \beta\alpha }(v)T_{u,\alpha}  | f -T_{v,\alpha}f |^2 dudv.
 \end{eqnarray*}
It is easy to see that $\|T_{u,\alpha}  | f -T_{v,\alpha}f |^2\|\leq (1+|\ln \frac uv|)\|f\|_{BMO(L^\alpha)}^2$, so we get (\ref{alpha<beta}) from (\ref{lnuv2}). 

For the rest of this proof, we   use $\Gamma$ for $\Gamma_{L^\beta}$, $T_t$ for $T_{t,\beta}$ and $P_t$ for $T_{t,\frac\beta2}$  to simplify the notation. Since $T_t$ has the quasi monotone property (\ref{csf}), we have 
\begin{eqnarray}
P_t=\int_0^\infty T_u\phi_{t,\frac12}(u)du\geq \int_0^{t^2} \left(\frac u{t^2}\right)^{K_1}T_{t^2} \phi_{t,\frac12}(u)du\geq \frac1{100K_1}T_{t^2}.
\end{eqnarray} We now prove (\ref{BMO<bmo}). Note
\begin{eqnarray*}  
 \|T_{t}|f -T_{t}f|^2 \|&=& \| T_{t}|f-T_tf|^2-|T_{t}f-T_tT_ t f|^2 +|T_{t}f-T_tT_ t f|^2\|\\
 &\leq&\|f-T_tf\|^2_{bmo(L^\beta)}+\|T_{t}f-T_{2t} f\|^2.
 \end{eqnarray*}
 Let $\gamma=2^{\frac1{K_1}}$   and $S=2T_t-T_{\gamma t}$. Then $S$ is a unital completely positive map because of (\ref{csf}). We have 
 \begin{eqnarray*}
 |T_{t}f-T_{\gamma t} f|^2 +|Sf-T_{t}f|^2
 &=&-2 |T_tf|^2 +|T_{\gamma t}f|^2+|Sf|^2\\
 &\leq&-2|T_tf|^2 +T_{\gamma t}|f|^2+ S|f|^2\\
 &\leq&-2|T_tf|^2 +2T_{t}|f|^2  \\
 &\leq& 2\|f\|^2_{bmo(L^\beta)}.
 \end{eqnarray*}
 We get  by the triangle inequality that
 $$\|T_{t}f-T_{2t} f\| \leq K_1\sup_s\|T_sf-T_{\gamma s}f\|\leq \sqrt2K_1\|f\|_{bmo(L^\beta)}. $$
 Therefore,
$$ \|f\|_{BMO (L^\beta)} \leq \sqrt{4+2K^2_1}\|f\|_{bmo (L^\beta)}.$$

To prove (\ref{BMO=bmo}), we note that the $\Gamma^2\geq0$ assumption for $L$ passes to $L^\alpha$ by Lemma \ref{sub}. The inequality $\|f\|_{bmo}\leq (2+\sqrt2\|f\|_{BMO})$ is proved in \cite[Proposition 2.4]{JM12} assuming the $\Gamma^2\geq0$ criterion. Together with (\ref{BMO<bmo}), we get $\|f\|_{BMO(L^\alpha)}\simeq \|f\|_{bmo(L^\alpha)}$. We now show the second equivalence in (\ref{BMO=bmo}). Note,
\begin{eqnarray*}
 \int_0^tT_{t-s}\Gamma(T_sP_{\sqrt t}f)ds&=&\int_0^t T_{t-s}\Gamma\left(\int_0^\infty \phi_{\sqrt t,\frac12}(v)T_vT_sfdv\right)ds\\
 &\leq&\int_0^\infty \phi_{\sqrt t,\frac12}(v)\int_0^t T_{t-s}\Gamma(T_vT_sf)dsdv\\
 &\leq&\int_0^\infty \phi_{\sqrt t,\frac12}(v)\int_0^t T_{t+v-\frac{t+v}ts}\Gamma(T_{\frac{t+v}ts}f)dsdv\\
(u={\frac{t+v}ts}) &\leq&\int_0^\infty \phi_{\sqrt t,\frac12}(v)\frac{t}{t+v}\int_0^{t+v} T_{t+v-u}\Gamma(T_{u}f)dsdv\\
 (\ref{Meyeralpha})&=&\int_0^\infty \phi_{\sqrt t,\frac12}(v)\frac{t}{t+v}(T_{t+v}|f|^2-|T_{t+v}f|^2)dv\\
 &\leq&\int_0^\infty \phi_{\sqrt t,\frac12}(v)\frac{t}{t+v}\|f\|_{bmo(L^\beta)}^2dv<\frac56\|f\|_{bmo(L^\beta)}^2.
\end{eqnarray*}
We then have
\begin{eqnarray*}
 &&(T_{t}|f|^2-|T_{t}f|^2)^\frac12\\
 &\leq&  (T_{t}|f-P_{\sqrt t}f|^2-|T_{t}f-T_tP_{\sqrt t}f|^2)^\frac12+(T_{t}|P_{\sqrt t}f|^2-|T_tP_{\sqrt t}f|^2)^\frac12\\
 &\leq&100K_1(P_{\sqrt t}|f-P_{\sqrt t}f|^2 )^\frac12+\sqrt\frac56\|f\|_{bmo(L^\beta)}\\
 &\leq&100K_1 \|f\|_{bmo(L^{\frac\beta2})}+\sqrt\frac56\|f\|_{bmo(L^\beta)},
 \end{eqnarray*}
 so
 $$\|f\|_{bmo(L^\beta)}\leq 1200K_1 \|f\|_{bmo(L^{\frac\beta2})}.$$
Therefore, $$\|f\|_{BMO(L^\beta)}\leq 10000K_1^2 \|f\|_{BMO(L^{\frac\beta2})}.$$
Applying (\ref{alpha<beta}), we have $\|f\|_{BMO(L^\alpha)}\simeq  \|f\|_{BMO(L^{\beta})}$ for all $0<\beta,\alpha<1$.
\end{proof}

\noindent {\bf Remark.} The equivalence (\ref{BMO=bmo}) fails for $\alpha=1$ in general. See   Section 4, Example 2.

\section{Imaginary powers and $H^\infty$-calculus}

\subsection{$H^\infty$-calculus.} 
Let us review some  definitions and basic facts about $H^\infty$-calculus. We refer the readers to \cite{CDMY96,JMX06,Ha06} for details.
For $0<\theta<\pi$, let $S_\theta$ be the following open sector of the complex plane:
$$S_\theta=\{z\in {\Bbb C}, |\arg z|<\theta\}.$$
Recall that  we say a  closed   operator $A$ on a Banach space $X$ is a {\it sectorial } operator of type $\omega<\pi$ if the spectrum of $A$ is contained in $\overline{S}_\omega$, the closure of $S_\omega$, 
and for any $\theta, \omega<\theta<\pi, z\notin S_\theta$, there exists $c_\theta$ such that
$$\|z(z-A)^{-1}\|\leq c_\theta.$$ 
We will assume that the domain of $A$ is dense  in $X$ (or weak* dense in $X$ when $X$ is a dual space).
We may also  assume that $A$ has dense range and is one to one by considering $A+\varepsilon$ (see \cite[Lemma 3.2, 3.5]{JMX06}).

Let $H^\infty(S_\eta)$ be the space of all bounded analytic functions on $S_\eta$ and $H_0^\infty(S_\eta) $ be the subspace of the functions $\Phi\in H^\infty(S_\eta)$ with an extra decay property that
$$|\Phi(z)|\leq \frac {c|z|^r}{(1+|z|)^{2r}},$$
for some $c,r>0$.
Then for any $\Phi\in H_0^\infty(S_\eta)$, and $\theta>\eta$,
\begin{eqnarray}\label{PhiA}
\Phi(A)=\frac1{2\pi i}\int_{\gamma_\theta}\Phi(z)(z-A)^{-1}dz
\end{eqnarray}
  is a well defined bounded operator on $D(A)$ and its (weak*) extension is bounded on $X$. Here $\gamma_\theta$ is the boundary of $S_\theta$ oriented counterclockwise.
For general $\Phi\in H^\infty(S_\eta)$, set
\begin{eqnarray}\label{Phi(A)}
\Phi(A)=\psi(A)^{-1}(\Phi\psi)(A),
\end{eqnarray}
with  $\psi(z)=\frac z{(1+z)^2}$. 
It turns out that the so defined $\Phi(A)$ is   a closed (weak*) densely defined operator, which may not be bounded, and it coincides with $\Phi(A)$ defined as in (\ref{PhiA}) for $\Phi\in H_0^\infty(S_\eta)$.
Moreover, these definitions are consistent with the definitions in the ``older'' functional calculus.

\begin{definition}
We say a (weak*) densely defined sectorial operator $A$ of type $\omega$  has bounded $H^\infty(S_\eta)$-calculus, $ \omega<\eta<\pi$, if the map $\Phi(A)$ extends to a bounded operator on $X$ and there is a constant $C$ such that
\begin{eqnarray}\label{Hinfty}
 \|\Phi(A)\|\leq C\|\Phi\|_{H^\infty(S_\eta)}
 \end{eqnarray} for any bounded analytic function $\Phi\in H^\infty(S_\eta)$.
\end{definition}
 
\noindent {\bf Remark.} Suppose a densely defined sectorial $A$ has bounded $H^\infty(S_\eta)$-calculus on $Y$ and suppose  $Y$ is a weak* dense subspace of a dual Banach space $X$. 
Then the weak* extension of $\Phi(A)$ onto $X$, still denoted by $\Phi(A)$, is bounded  and satisfies (\ref{Hinfty}) with the same constant. So a 
weak* dense sectorial operator $A$ has $H^\infty$-calculus on $X$ if and only if it has $H^\infty$-calculus on   the norm closure of $D(A)$.
\medskip

   The negative infinitesimal generator $L$ of any uniformly bounded  (weak*) strong continuous semigroup on a dual Banach space $X$ is actually a   (weak*) densely defined sectorial operator of type $\frac \pi2$   and $L^\alpha$ is of type $\frac {\alpha\pi}2$ on $X$. Cowling, Duong, and Hiebe $\&$ Pr\"uss (\cite{Co83, Du89,HP98}) prove that the negative infinitesimal generator of a semigroup of positive contractions on $L^p, 1<p<\infty$ always has the bounded $H^\infty(S_\eta)$-calculus for any $\eta>\frac\pi2$. One cannot hope to extend this to    $p=\infty$. We will prove that  the associated  BMO$(\sqrt L)$  space is   a good alternative, as desired. 

\begin{lemma}\label{Mc}
Suppose $A$ is a densely defined sectorial operator of type $\omega<\pi/2$ on a Banach space $X$. Assume 
$\int_0^\infty Ae^{-tA}a(t)dt$ is bounded on $X$ with norm smaller than $C$ for any  function $a(t)$ with values in $\pm1$.
Then $A$ has a bound $H^\infty(S_\eta^0)$ calculus for any $\eta>\pi/2$. 
\end{lemma}
\begin{proof} This  is a consequence of \cite[Example 4.8]{CDMY96} by setting $a(t)$ to be the sign of $  \langle Te^{-tT}u, v\rangle$ for any pair $(u,v)$ in a dual pair $( X, Y)$.
\end{proof}

We are going to prove that the negative generator $L$ of a semigroup of positive contractions satisfies the assumptions of Lemma \ref{Mc}. We follow an idea of E. Stein and consider   scalar valued functions $a(t)$ such that
 \begin{eqnarray}\label{a(t)}
 s\int_s^\infty\frac {|a(v-s)|^2}{v^2}dv\leq c^2_a ,
 \end{eqnarray}
 for all $s>0$ and some constant $c_a$. Define $M_a$ by
\begin{eqnarray}\label{Ma}
M_a(f)=\int_0^\infty a(t)\frac {\partial T_{t,\alpha}f}{\partial t}dt=\int_0^\infty a(t) L^\alpha T_{t,\alpha}f dt,
\end{eqnarray}
for $f\in L^\infty, 0<\alpha<1$. For now, we assume  $a$ is supported on a compact subset of $(0,\infty)$ so we do not worry about the convergence of the integration.
\begin{lemma}\label{keybmo} Assume that $L $ generates a weak* continuous semigroup of positive contractions on $L^\infty$ satisfying the $\Gamma^2\geq0$ criterion (\ref{Gamma_2}). We have
\begin{eqnarray}
\|M_a(f)\|_{bmo (L^\alpha)} &\leq& \frac{cc_a}{(1-\alpha)^2}\|f\|_{bmo (L^\alpha)},\label{bmobmo}\\
\|M_a(f)\|_{BMO (L^\alpha)} &\leq& \frac{cc_a}{(1-\alpha)^3}\|f\|_{BMO (L^\alpha)},\label{BMOBMO}\\
\|M_a(f)\|_{BMO (L^\alpha)} &\leq&  \frac{cc_a}{(1-\alpha)^2}\|f\|_{ L^\infty},\label{LinftyBMO}
\end{eqnarray}
for any $f\in L^\infty, 0<\alpha<1$.
\end{lemma}

\begin{proof} We consider the case $\alpha\geq\frac34$ only. The case $\alpha< \frac34$ is easier and follows from this case by subordination. 
  Recall that the $\Gamma^2\geq0$ assumption for $L$ passes to $L^\alpha$ by Lemma \ref{sub}, and $T_{t,\alpha}(L^\infty)\subset D(L^{2\alpha}), L^{2\alpha}T_{t,\alpha}=\partial_t^2T_{t,\alpha}$. In this proof, we use $\Gamma$ for $\Gamma_{L^\alpha}$ the gradient form associated with $L^\alpha$, $T_t$ for $T_{t,\alpha}$ and $P_t$ for $T_{t,\frac\alpha2}$  to simplify the notation.  Let $r=\frac1{1-\alpha}>4$. We have that 
\begin{eqnarray*}
\int_0^tT_{rt-s}\Ga(T_sf)ds&=&\int_0^tT_{rt-s}\Ga\left(\int_s^\infty L^\alpha T_vfdv\right)ds\\
&\leq&\int_0^tT_{rt-s}\int_s^\infty\Ga( L^\alpha T_vf)v^\frac32dv\int_s^\infty v^{-\frac32}dvds\\
&=&\int_0^tT_{rt-s}\int_s^\infty\Ga( L^\alpha T_vf)v^\frac32dv2s^{-\frac12}ds\\
&=&\int_0^\infty \int_0^{t\wedge v}2s^{-\frac12}T_{rt-s}ds\Ga( L^\alpha T_vf)v^\frac32 dv
\end{eqnarray*}

Let $S_v= \int_0^{t\wedge v}2s^{-\frac12}T_{rt-s}ds$.
So by Lemma \ref{Meyer} and the $\Gamma^2\geq0$ criterion,
\begin{eqnarray*}
\|f\|^2_{bmo }&=&\sup_t\left\|\int_0^{rt}T_{rt-s}\Ga(T_sf)ds\right\|\\
 &\leq& \left\|\sup_t \int_0^{rt}T_{rt-\frac sr}\Ga(T_{\frac sr}f)ds\right\|\\
&=&\left\| \sup_t r\int_0^{t}T_{rt-  s}\Ga(T_{s}f)ds\right\|\\
&\leq&\sup_t r\left\|\int_0^\infty   S_v \Ga( L^\alpha T_v)v^\frac32  dv\right\|.\end{eqnarray*}

So,
\begin{eqnarray*}
\frac1r\| M_af\|^2_{bmo} &\leq&\left\|\int_0^\infty   S_v\Ga( L^\alpha T_vM_a(f))v^\frac32dv\right\| \\
 &=&
\left \|\int_0^\infty S_v\Ga( T_v\int_0^\infty a(u)  L^{2\alpha}T_u fdu)v^\frac32 dv\right\| \\
&=&\left\|\int_0^\infty S_v \Ga(\int_0^\infty a(u)  L^{2\alpha}T_{u+v} fdu)v^\frac32  dv\right\| \\
 &=&
  \left\|\int_0^\infty v^\frac32S_v \Ga(\int_v^\infty a(u-v)\frac 1uu   L^{2\alpha}T_u fdu)dv\right\| \\
\text{({\rm Inequality}\ (\ref{cauchy}))} &\leq& \left\|\int_0^\infty v^\frac32S_v\bigg(\int_v^\infty \frac{%
|a|^2}{u^2}du\int_v^\infty \Ga\left(u L^{2\alpha}T_u 
f\right)du\bigg)dv\right\| \\
 &\leq&
  c_a^2\left\|\int_0^\infty S_v\bigg(\int_v^\infty  \Ga\left(u L^{2\alpha}T_u f\right)du\bigg) v^\frac12  dv\right\| \\
&=&
  c_a^2\left\|\int_0^\infty \int_0^{u\wedge t} v^\frac12S_v  dv\Ga(u  L^{2\alpha}T_{u} f)du\right\|.
\end{eqnarray*}
Note $K_1\leq r$ and $\sup_{r>4}(\frac2{1+\alpha}\frac r{r-1})^r\leq c$. By (\ref{csf}), we have, for $u\leq t$, 
\begin{eqnarray*}
  \int_0^{t\wedge u} v^\frac12 S_v dv& \leq& \int_0^{t\wedge u}v^\frac12\int_0^{t\wedge v}s^{-\frac12}T_{ \frac {1+\alpha }2(rt-u)}\left(\frac 2{1+\alpha }\cdot\frac {r}{r-1}\right)^{r}dsdv\\
  &\leq& cT_{\frac {1+\alpha }2(rt-u)}  {t^2\wedge u^2}.
  \end{eqnarray*}
  Applying  (\ref{cdsf22}), we get
  \begin{eqnarray*}
  \left\|\int_0^t \int_0^{u\wedge t} v^\frac12S_v  dv\Ga(u  L^{2\alpha}T_{\frac u2} T_{\frac u2}f)du\right\|&\leq&
  cr^2\left\|\int_0^t \frac{T_{\frac {\alpha u}2 }}{u^2}\int_0^{u\wedge t} v^\frac12S_v  dv\Ga( T_{\frac u2}f)du\right\|\\
  &\leq& cr^2\left\|\int_0^t  T_{\frac{(1+\alpha)rt}2-\frac u2}\Ga( T_{\frac u2}f)du\right\| \\
  &\leq&
  cr^2 \left\|\int_0^{\frac t2}  T_{\frac t2-s}\Ga( T_{s}f)ds\right\|\leq cr^2\|f\|^2_{bmo }.
\end{eqnarray*}
For $\alpha^{-n}t<u\leq \alpha^{-n-1}t,n\geq 0$, we use 
$$
  \int_0^{t\wedge u} v^\frac12 S_v dv \leq \int_0^{t\wedge u}v^\frac12\int_0^{t\wedge v}2s^{-\frac12}T_{rt-t}\left(\frac {r}{r-1}\right)^{r}dsdv
  \leq cT_{rt-t}  {t^2\wedge u^2}.$$
Similar to  (\ref{cdsf22}), we get $\Gamma(u^2L^{2\alpha}T_{ \alpha^{-n}t}f)\leq cr^2T_{ 2\alpha^{-n}t-u }\Gamma(f)$ because $\frac{r-1}{r-2}=\frac1{2-\alpha^{-1}}\leq\frac{ \alpha^{-n}t }{ 2\alpha^{-n}t-u} \leq1$. So
$$\Ga(u  L^{2\alpha}T_{ \alpha^{-n}t} T_{  u-\alpha^{-n}t}f)|\leq  c\frac{r^2 }{u^2}T_{ 2\alpha^{-n}t-u } \Ga( T_{ u-\alpha^{-n}t}f)
$$
 Therefore, \begin{eqnarray*}
 && \left\|\int_{\alpha^{-n}t}^{\alpha^{-n-1}t} \int_0^{u\wedge t} v^\frac12S_v  dv\Ga(u  L^{2\alpha}T_{\frac u2} T_{\frac u2}f)du\right\|\\
  &\leq&
  cr^2 \left\| \int_{\alpha^{-n}t}^{\alpha^{-n-1}t} \frac{  {t^2\wedge u^2}}{u^2}T_{ 2\alpha^{-n}t-u } \Ga( T_{ u-\alpha^{-n}t}f)du\right\|\\
 &=& cr^2\alpha^{2n}\left\| \int_{0}^{\alpha^{-n-1}t(1-\alpha)} T_{ \alpha^{-n}t-s } \Ga( T_{ s}f)ds\right\|\\
  &\leq&
  cr^2\alpha^{2n}  \|f\|^2_{bmo }.
\end{eqnarray*}
Summing up for $n\geq 0$, we get\begin{eqnarray*}
  \left\|\int_{t}^{\infty} \int_0^{u\wedge t} v^\frac12S_v  dv\Ga(u  L^{2\alpha}T_{\frac u2} T_{\frac u2}f)du\right\| \leq
  cr^3  \|f\|^2_{bmo }.
\end{eqnarray*}
Combining the estimates above, we conclude that
$$\|M_a(f)\|_{bmo (L^\alpha)}\leq cc_ar^2\|f\|_{bmo (L^\alpha)}.$$
Applying (\ref{BMO=bmo}), we actually get 
$$\|M_a(f)\|_{BMO (L^\alpha)}\leq c  c_ar^3 \|f\|_{BMO (L^\alpha)}\leq c  c_a r^3\|f\|_{L^\infty}.$$
But we wish to get a better estimate.
Note \begin{eqnarray*}
(T_t-T_{2t})M_a(f)&=&\int_0^\infty a(s)\partial_s (T_{t+s}-T_{2t+s})fds\\
&=&\int_t^\infty a(s-t)\partial_s (T_{s}-T_{t+s})fds\\
&\leq&\left(\int_t^\infty \frac{|a(s-t)|^2}{s^2}ds\right)^\frac12\left( \int_t^\infty s^2\left|\int_0^t\partial^2_s  T_{v+s}fdv\right|^2ds\right)^\frac12\\
&\leq&{c_a}\left( \int_t^\infty {s^2} \int_0^t |\partial^2_sT_{v+s}f|^2dvds\right)^\frac12\\
(by\ (\ref{cdsf2}))&\leq&\frac{25c_a}{(1-\alpha)^2}\left( \int_t^\infty {s^{-2}} \int_0^t  T_{\alpha(v+s)}|f|^2dvds\right)^\frac12.
 \end{eqnarray*}
 Therefore
 \begin{eqnarray*}
\|(T_t-T_{2t})M_a(f)\|_{L^\infty}\leq \frac{25c_a}{(1-\alpha)^2}\|f\|_{L^\infty},
\end{eqnarray*}
and hence
$$\|M_a(f)\|_{BMO (L^\alpha)}\leq   \|M_af\|_{bmo (L^\alpha)}+\sup_t \|(T_t-T_{2t})M_a(f)\|_{L^\infty}\leq c r^2 c_a \|f\|_{L^\infty}.$$
\end{proof}

 Given $f\in L^\infty, g\in H^1(L^\alpha)$, let $\tilde a(t)=sign \langle L^\alpha T_{t,\alpha}f,g \rangle  a(t)$. Then $\tilde a$ satisfies (\ref{a(t)}) if $a$ does. We have from Lemma \ref{keybmo} that
\begin{eqnarray*}
\int_0^\infty |\langle a(t)L^\alpha T_{t,\alpha}f,g \rangle|dt&=&\lim_{N,M\rightarrow \infty}\int_{\frac1M}^N |\langle a(t)L^\alpha T_{t,\alpha}f,g \rangle|dt\\
&=&\lim_{N,M\rightarrow \infty}\left\langle \int_{\frac1M}^N \tilde a(t)L^\alpha T_{t,\alpha}fdt,g\right\rangle\\
&\leq&cc_a \|M_{\tilde a} f\|_{BMO(L^\alpha)}\|g\|_{H^1}\\
&\leq&\frac{cc_a}{(1-\alpha)^2} \| f\|_{ L^\infty}\|g\|_{H^1}.
\end{eqnarray*}

This shows that $\lim_{N,M\rightarrow \infty}\int_{\frac1M }^N \langle a(t)L^\alpha T_{t,\alpha}f,g \rangle dt$ exists and  $ \int_{\frac1M}^N   a(t)L^\alpha T_{t,\alpha}f dt$ weak* converges in BMO$(L^\alpha)$ as $N,M\rightarrow \infty$.  So the integration in  (\ref{Ma})  weak* converges and $M_a$ is well defined for all   $f\in L^\infty$ and $a(t)$ satisfying (\ref{a(t)}). The weak* extension of $M_a$ is then a bounded map from BMO$(L^\alpha)$ to BMO$(L^\alpha)$.  

\begin{theorem}  Suppose $a(t)$ satisfies (\ref{a(t)}). $M_a$ extends to a bounded operator from BMO$(L^\alpha)$ to BMO$(L^\alpha)$  for $0<\alpha<1$. The estimates are as in Lemma \ref{keybmo}.
\end{theorem}


\begin{theorem}\label{main3}
Suppose $T_t=e^{-t{L}} $ is a weak* continuous  semigroup of positive contractions  on $L^\infty$ satisfying the $\Gamma^2\geq0$ criterion. Then $L$ has a complete bounded $H^\infty(S_\eta)$ calculus on $BMO(\sqrt L)$ for any $\eta> \frac\pi2$. 
\end{theorem}

\begin{proof} Given $\alpha\in (\frac12,1)$, let $Y^\alpha$ be the norm closure of $D(L)$ in BMO$(L^\alpha)$. It is easy to check that $T_{t,\alpha}=e^{-t{ L}^\alpha} $  are contractions on   $Y^\alpha$.  Then $L^\alpha$ is a   densely defined sectorial operator of type $\frac\pi2$ in $Y^\alpha$.   Lemma \ref{Mc} and Lemma \ref{keybmo} imply that $L^\alpha$ has a   bounded $H^\infty(S_\eta)$ calculus on $Y^\alpha$ for any $\eta> \frac\pi2$.
Note  $\Phi(z)=\Psi(z^\frac1\alpha)\in S_ \eta $ if $\Psi\in S_\frac\eta\alpha$ and $\Phi(L^\alpha)=\Psi(L)$. We  conclude that $L$ has a   bounded $H^\infty(S_\eta)$ calculus on $Y^\alpha$ for any $\eta> \frac\pi{2\alpha}$. Given $\theta>\frac\pi2$, choose $\frac12<\alpha<1$ so that $\alpha\theta>\frac\pi2$. Then $L$ has a  bounded $H^\infty(S_\theta)$ calculus on $Y^\alpha$. Lemma \ref{equibmo} then implies that $L$ has a   bounded $H^\infty(S_\theta)$ calculus on $Y^\frac12\simeq Y^\alpha$ and on $BMO(\sqrt L)$ for any $\theta>\frac\pi2$, since $Y^\frac12$ is weak * dense in BMO$(\sqrt L)$ and $\Phi(L)$ is the weak* extension of its restriction on $Y^\frac12$ by definition.  The same argument applies to $id\otimes L$. We then obtain the completely bounded $H^\infty(S_\eta)$ calculus as well.
\end{proof}

\subsection{Imaginary Power and Interpolation.} Given $0<\alpha<1$, choose $\frac\pi2<\theta<\frac\pi{2\alpha}$. By (\ref{Phi(A)}), we have   the identities
\begin{eqnarray}
L^\alpha e^{-tL^\alpha}&=&\frac1{2\pi i}\psi(L)^{-1}\int_{\gamma_\theta} z^\alpha \psi(z) e^{-tz^\alpha} (z-L)^{-1}dz,\\
L^{i\alpha s}&=& \frac1{2\pi i}\psi(L)^{-1}\int_{\gamma_\theta} z^{i\alpha s}\psi(z)(z-L)^{-1} dz.
\end{eqnarray} 
Note 
\begin{eqnarray} 
z^{i\alpha s}&=&\Gamma(1-is)^{-1}\int_0^\infty t^{-is}z^\alpha e^{-tz^\alpha}dt.
\end{eqnarray}
Since these integrals converge absolutely, we can exchange the order of the integrations and get
\begin{eqnarray}
L^{i\alpha s}&=&\Gamma(1-is)^{-1}\int_0^\infty t^{-is}L^\alpha e^{-tL^\alpha}dt.
\end{eqnarray}
The inequality (\ref{LinftyBMO}}) of Lemma \ref{keybmo} implies that
$$\|L^{i\alpha s}f\|_{  BMO(L^\alpha)}\leq \frac c{(1-\alpha)^2}\Gamma(1-is)^{-1}\|f\|_{L^\infty}.$$
Thus for $\alpha>\frac12,$
\begin{eqnarray*}
\|L^{i s}f\|_{ BMO(L^\alpha)}&\leq& c\Gamma\left(1-i\frac s{\alpha}\right)^{-1}\|f\|_{L^\infty}\\
&\leq& \frac c{(1-\alpha)^2(1+|s|)^{\frac12}}\exp\left(\frac {\pi |s|}{2\alpha}\right)\|f\|_{L^\infty}.
\end{eqnarray*}
Choosing $\alpha=\frac {|s|}{|s|+1}$ for $s$ large, we get
\begin{eqnarray}
\|L^{i s}f\|_{ BMO(L^\alpha)}\leq  c(1+|s|)^{\frac32}\exp\left(\frac {\pi |s|}{2}\right)\|f\|_ {L^\infty}.\label{Lis}
\end{eqnarray}
The same estimate holds with bmo$(L^\alpha)$-norms putting on both sides of (\ref{Lis}) because we can apply the same argument to   (\ref{bmobmo}) of Lemma \ref{keybmo} instead of (\ref{LinftyBMO}). 
Applying the inequality (\ref{alpha<beta}) to (\ref{Lis}), we  get
\begin{eqnarray}
\|L^{i s}f\|_{ BMO(\sqrt L)}\leq  c(1+|s|)^{\frac32}\exp\left(\frac {\pi |s|}{2}\right)\|f\|_ {L^\infty}.\label{Lis12}
\end{eqnarray}
Applying the inequalities (\ref{BMOBMO}), (\ref{BMO1/2})  instead of (\ref{LinftyBMO}), (\ref{alpha<beta}), we will have similarly 
\begin{eqnarray}
\|L^{i s}f\|_{ BMO(\sqrt L)}\leq  c(1+|s|)^{\frac92}\exp\left(\frac {\pi |s|}{2}\right)\|f\|_ {BMO(\sqrt L)}.\label{LisBMO12}
\end{eqnarray}

\begin{definition} We say a weak* continuous   semigroup of positive contractions is a {\it symmetric  Markov} semigroup if  $\langle  T_tf,g\rangle=\langle  f,T_tg\rangle$ for $f\in L^\infty,g\in L^1$ and it admits a standard Markov dilation in the sense of \cite[page 717]{JM12}.
\end{definition}

\noindent{\bf Remark.}  The Markov dilation assumption in the above definition holds automatically in many cases. In the commutative case (i.e the underlying von Neumann algebra   $\M=L^\infty(M)$), this is due to Rota (see \cite[page 106, Theorem 9]{St70}). Therefore every weak* continuous semigroup of unital symmetric positive contractions is automatically a symmetric Markov semigroup.  In \cite{R08} it is proven that this is the case  for convolution semigroups on group von Neumann algebras.  In \cite{Da,JRS} it is proven that this holds for the finite von Neumann algebras case. The case of a general semifinite von Neumann algebra is conjectured but there has not been a written proof.

 \begin{lemma} \label {JM12}([JM12]) Assume that $T_t=e^{-tA}$ (e.g. $A=L^\alpha$) is a    symmetric Markov semigroup on a semifinite von Neumann algebra $\M$. Then, the following interpolation result holds
$$ [BMO(A),L^1_0(\M)]_{\frac{1}{p}} = L^p_0(\M) $$
for
 $1<p<\infty$.   Here $L^p_0(\M)=L^p(\M)/ker A$.  
\end{lemma}
\medskip

Since $\|L^{i s}\|_{L^2\rightarrow L^2}=1$ if $L$ generates a symmetric Markov semigroup, by interpolation, we get from (\ref{Lis12}) the following result.
\begin{corollary}\label{correctLis}
Suppose $T_t=e^{-t{L}} $ is a symmetric Markov semigroup of operators on a semifinite von Neumann algebra $\M$ and satisfies the $\Gamma^2\geq0$ criterion. Then, $L$ has the completely bounded $H^\infty(S_\eta )$-calculus on $L^p$ for any $\eta>\omega_p=|\frac\pi2-\frac\pi p| , 1<p<\infty$ and 
\begin{eqnarray}\label{Lisp}
\|L^{i s}\|_{L^p\rightarrow L^p}\leq  c(1+|s|)^{|\frac32-\frac3p|}\exp\left(|\frac{\pi s}2 -\frac{\pi s}p|\right),
\end{eqnarray}
for all $1<p<\infty$.
\end{corollary}
{\bf \it Remark.} Let us point out that  the left-hand side of the inequality  on \cite[line 4, page 728]{JM12} misses a ``$\frac12$". It should be $\|L^{\frac{is}2}f\|$ instead of  $\|L^{{is}}f\|$, because Theorem 3.3 of \cite{JM12} is for the semigroup generated by $\sqrt L$. So the estimate of the constants $c_{s,p}$ given in \cite[Corollary 5.4]{JM12} is not correct. Also \cite{JW17} contains a similar estimate to (\ref{Lisp}) without assuming the $\Gamma^2\geq0$ criterion. Their method is   the transference principle and works for $L^p$ only.

Junge, Le Merdy, and Xu (\cite{JMX06})   studied the   $H^\infty$-calculus in the noncommutative setting. In particular, they prove a $H^\infty(S_\eta)$-calculus property of ${  L}:\lambda_g\mapsto |g|\lambda_g$   on $L^p(\hat {\mathbb F}_n)$    for all $1<p<\infty, \eta> |\frac\pi2-\frac\pi p| $. Here $L^p(\hat {\mathbb F}_n)$  is the noncommutative $L^p$-space associated with the free group von Nuemann algebra. 


\section{Examples}

The ``$\Gamma^2\geq0$" criterion  is known to be satisfied by a large class of semigroups including the heat, Ornstein-Uhlenbeck, and Jacobi semigroups (see \cite{Ba96}). The results proved in this article apply to all of them. The main example in the noncommutative setting, is the   semigroup of operators on a group von Neumann algebra, generated from a conditionally negative function on the underlying group (see Example 4). We will analyze a few of them in the following.

 \begin{example} Let $-L=\Delta$ be the Laplace-Beltrami operator on a complete Riemannian manifold with  nonnegative Ricci curvatures. Then  the associated heat semigroup $T_t=e^{-tL}$  is symmetric Markovian and satisfies the $\Gamma^2\geq0$ criterion. 
  All the theorems of this article hold for $L$, and it has bounded $H^\infty(S_\eta)$  calculus on  $BMO(\sqrt L)$ for any $\eta>\frac\pi2$.
  
  In the special case that $L=-\partial_x^2$ the Laplacian on  Euclidean space   ${\Bbb R}^n$, the BMO$(L )$,  bmo$(L ),  $ and BMO$(\sqrt L)$ spaces are all equivalent to  the classical BMO space of all functions $f\in L^1({\Bbb R}^n,\frac1{1+|x|^2}dx)$ with a finite BMO norm,
 \begin{eqnarray*}
 \|f\|_{BMO({\Bbb R}^n)}=\sup_{B \subset {\Bbb R}^n} (E_B|f-E_Bf|^2)^\frac12<\infty.
 \end{eqnarray*}
 Here the supremum runs on all balls (or cubes) in ${\Bbb R}^n$ and $E_B=\frac1{|B|}\int_Bfdx$ denotes the mean value operator. This can be verified by the integral representation  of $T_t, T_{t,\frac12}$,
the convexity of $|\cdot|^2$ and the fact that $|E_Bf-E_{kB}f|\lesssim  \log k\|f\|_{BMO({\Bbb R}^n)}$. By Lemma \ref{equibmo} we then get the  equivalence between BMO$({\Bbb R}^n)$ and BMO$(L^\alpha)$ for all $0<\alpha\leq 1$.

 \end{example}
\begin{example} Let $L=\partial_x$ on ${\Bbb R}$. Then   $T_t=e^{-tL}$ is the translation operator sending $f(\cdot)$ to $f(\cdot-t)$. It is a Markov semigroup and the $\Gamma^2\geq0$ criterion holds trivilly. The BMO$(L)$ space is equivalent to $L^\infty_0$ and the bmo$(L)$ (semi)norm vanishes. 
 For any $0<\alpha<1$,  BMO$(L^\alpha)$ is equivalent to the classical BMO$({\Bbb R}^n)$ space.
 Indeed, by the subordination formula, we get
 the following integral representation for $T_{t,\frac12}=e^{-t\sqrt L}$: 
  \begin{eqnarray*}
T_{t,\frac12}f(x) =\frac1{2\sqrt \pi}\int_0^\infty f(x-s) te^{-\frac{t^2}{4s}}s^{-\frac32} ds. 
 \end{eqnarray*}
 From this, it is easy to check that, for $I_{x,k}=[x-2^k\frac{t^2}4,x-2^{-k}\frac{t^2}4], k\in {\Bbb N}$,
   \begin{eqnarray*}
c^{-1}E_{I_{x,1}}|f| \leq T_{t,\frac12}|f|(x) \leq c\sum_{k} 2^{-\frac k2}E_{I_{x,k}}|f|. 
 \end{eqnarray*}
After an elementary calculation  and using the fact that $$|E_Bf-E_{kB}f|\lesssim  \log k\|f\|_{BMO({\Bbb R}^n)},$$ one can see that $ \|\cdot\|_{BMO(\sqrt L)}\simeq \|\cdot\|_{BMO({\Bbb R}^n)}$, thus $\|\cdot\|_{BMO( L^\alpha)}\simeq  \|\cdot\|_{BMO({\Bbb R}^n)}$ for all $0<\alpha<1$ by Lemma \ref{equibmo}. 

 By Theorem \ref{main3}, $L$ has $H^\infty(S_\eta)$-calculus on $BMO(\sqrt L)\simeq BMO({\Bbb R}^n)$ for any $\eta>\frac\pi2$.  It is easy to see that $$L^{is}=P_+e^{\frac{-s\pi}2}\Delta^{\frac {is}2}+P_-e^{\frac{s\pi}2}\Delta^{\frac {is}2}.$$ 
 So $L$ does not have $H^\infty(S_\theta)$-calculus on $BMO(L)\simeq L^\infty({\Bbb R})/{\Bbb C}$ for any positive $\theta$ and $$\|L^{is}\|_{BMO\rightarrow BMO}\simeq e^{\frac {\pi|s|}2}\|\Delta^{\frac {is}2}\|_{BMO\rightarrow BMO}$$ for $|s|$ large. This   indicates that it is better to  consider $BMO({\sqrt L})$ instead of $BMO(L)$ for the purpose of this article.

\end{example}
\begin{example} Let $-L=\frac {\partial_x^2}2-x\cdot\partial_x$ be the Ornstein-Uhlenbeck operator on $({\Bbb R}^n,e^{-{|x|^2}}dx)$.  Let  $O_tf=O_{t,1}=e^{-tL}.$ $O_t$ is a symmetric Markov semigroup with respect to the Gaussian measure $d\mu=e^{-|x|^2}dx$ and satisfies the $\Gamma^2\geq0$ criterion. Theorem \ref{main3} says that $L=-\frac {\partial_x^2}2+x\cdot\partial_x$ has bounded $H^\infty(S_\eta)$-calculus on $BMO(\sqrt L)$ for any $\eta>\frac\pi2$.

Mauceri and Meda (see \cite{MM07}) introduced the following BMO space for the Ornstein-Uhlenbeck semigroup
  \begin{eqnarray}\label{MM}
  \|f\|_{BMO(MM)}=\sup_{r_B\leq \min\{1,\frac1{|c_B|}\}}(E^\mu_B|f-E^\mu_Bf|^2)^\frac12,
  \end{eqnarray}
  with  $r_B,c_B$ the radius and the center of $B$, and $E^\mu_B=\frac1{\mu(B)}\int \cdot d\mu$ the mean value operator with respect to the Gaussian measure $d\mu$.  Note, for the balls $B$ satisfying $r_B\leq \min\{1,\frac1{|c_B|}\}$, we have the equivalence $E^\mu_B|f|\simeq E_B|f|$. One may replace  $E^\mu_B$ by   $E_B$, the mean value operator with respect to the Lebesque measure $dx$ in (\ref{MM}).  The resulted BMO norms are equivalent to each other.
From the integral presentation 
\begin{eqnarray}
O_t(f)=\frac1{(\pi-\pi e^{-2t})^\frac n2}\int_{{\Bbb R}^n} \exp\left(-\frac{|e^{-t}x-y|^2}{1-e^{-2t}}\right)f(y)dy,
\end{eqnarray}
one easily see that, for $t\leq4$ and $\sqrt t |x|\leq1$,  
 \begin{eqnarray}\label{Ot>EB}
  O_t|f| (x)&\geq& \frac1{(\pi-\pi e^{-2t})^\frac n2}\int_{B(x,\sqrt t)} \exp\left(-\frac{2|x-y|^2}{1-e^{-2t}}\right)f(y)dy\nonumber\\
  &\geq& c_nE_{B(x,\sqrt t)}|f|(x).
  \end{eqnarray}
  Note $E_{B(x,\sqrt t)}|f|\leq c_nE_{B(x,\sqrt s)}|f|$ for all $t<s<2t$. We then have from (\ref{Ot>EB}) that, for $O_{t,\frac12}=e^{-tL^\frac12}, t\leq1, tx\leq1$,  
   \begin{eqnarray*}
  O_{t,\frac12}|f| (x)=\int_0^\infty  O_s|f|(x)\phi_{t,\frac12}(s)ds\geq  \frac c{\sqrt t}\int_{t^2} ^{ 4t^2}  O_s|f|(x)ds\geq c_nE_{B(x,t)}|f|(x).
  \end{eqnarray*}
    We then easily get
   \begin{eqnarray}\label{MM<L}
   4O_{t,\alpha}|f-O_{t,\alpha}f|^2 (x) \geq c_nE_{B(x, t^\frac1{2\alpha})}|f-E_{B(x,  t^\frac1{2\alpha})} f(x)|^2(x),
  \end{eqnarray}
   by the convexity of $|\cdot|^2$, for $\alpha=\frac12,1$. Therefore,   $$\|\cdot\|_{BMO(MM)}\lesssim\|\cdot\|_{BMO(L)},\|\cdot\|_{bmo(L)},\|\cdot\|_{BMO(\sqrt L )} ,$$ and by Lemma \ref{equibmo},  $$\|\cdot\|_{BMO(MM)}\lesssim\|\cdot\|_{BMO(L^\alpha)} $$ for all $0<\alpha\leq 1$.  
  By Theorem \ref{main3}, the Ornstein-Uhlenbeck operator $ L=-\frac {\partial_x^2}2+x\cdot\partial_x$ has bounded $H^\infty(S_\eta)$ calculus from $L^\infty({\Bbb R}^n)$  to Mauceri-Meda's BMO$(MM)$ for any $\eta>\frac\pi2$.

Let $f(y)=\frac1{\sqrt{4\pi s}}   \exp(-\frac{|y|^2}{4s})$, with $s>100$. We have 
\begin{eqnarray*}
  &&(O_t|f|^2-|O_tf|^2)(x)\\
  &=&\frac1{4\pi\sqrt { (s+2v)   s}}\exp\left(-\frac{|e^{-t}x|^2}{2s+4v}\right)-\frac1{  {4\pi (s+v)}}\exp\left(-\frac{|e^{-t}x|^2}{2s+2v}\right)\\
  &=& (\frac1{4\pi\sqrt { (s+2v)   s}}-\frac1{  {4\pi (s+v)}})\exp\left(-\frac{|e^{-t}x|^2}{2s+4v}\right)\\
  &&\hskip 2cm +\frac1{  {4\pi (s+v)}}\left(\exp\left(-\frac{|e^{-t}x|^2}{2s+4v}\right)-\exp\left(-\frac{|e^{-t}x|^2}{2s+2v}\right)\right)\\
  &\lesssim&\frac1{s^3}+\frac1{s^2}\lesssim\frac1{  s^2}.
\end{eqnarray*}
On the other hand, for $v=\frac{1-e^{-2t}}4,v'=\frac{1-e^{-4t}}4$,
\begin{eqnarray*}
  (O_t f - O_{2t}f)(x) = \frac1{\sqrt{4\pi(s+v)}}e^{-\frac{|e^{-t}x|^2}{4s+4v}}-\frac1{\sqrt{4\pi(s+v')}}e^{-\frac{|e^{-2t}x|^2}{4s+4v'}}
  \end{eqnarray*}
  For $x^2=e^{2t}(4s+4v),t=10$, we get
  \begin{eqnarray*}
 |( O_t f - O_{2t}f)(x)|&\geq&| \frac1{\sqrt{4\pi(s+v)}}e^{-1}-\frac1{\sqrt{4\pi(s+v')}}e^{- \frac1{100}}|\\
 &\geq& \frac1{2\sqrt{4\pi(s+v')}} \geq \frac1{10\sqrt s}.
  \end{eqnarray*}
  So, $$\|f\|_{BMO(L)}\geq \sup_{t>0} \| O_t f - O_{2t}f\|_{L^\infty}\geq\frac{  \sqrt s} 5\|f\|_{bmo(L)}.$$
Therefore, the   BMO$(L)$ and bmo$(L)$-norms are not equivalent for the Ornstein-Uhlenbeck semigroup, by letting $s\rightarrow \infty$. This shows that one can not extend Lemma \ref{equibmo} to the case of $\alpha=1$.

\end{example}
\begin{example} Let $(G,\mu)$ be a locally compact unimodular group with its Haar measure. Let $\lambda_g, g\in G$ be the translation-operator on $L^2(G)$ defined as
$$\lambda_g(f)(h)=f(g^{-1}h) .$$
  The so-called group von Neumann algebra ${  L}^\infty(\hat G)$  is the weak* closure  in $B(L_2(G))$ of the operators $f=\int_G \hat f(g)\lambda_gd\mu(g)$ with $\hat f\in C_c(G).$
The canonical trace $\tau$ on $  L^\infty(\hat G)$ is defined as
$\tau f=\hat f(e)$.
If $G$ is abelian, then ${  L}^\infty(\hat G)$ is the canonical $L^\infty$ space of functions on the dual group $\hat G$. In particular, if $G={\Bbb Z}$, the integer group, then $\lambda_k=e^{ikt}, k\in {\Bbb Z}$ and ${  L}^p(\hat{\Bbb Z})=L^p({\Bbb T})$, the function space on the unit circle. Please refer to \cite{PX03} for details on noncommutative $L^p$ spaces.

Let $\varphi$ be a scalar valued function on $G$. We say $\varphi$ is {\it conditionally negative} if $\varphi(g^{-1})=\varphi(g)^*$ and
\begin{eqnarray}\label{CN}
 \sum_{g,h}\overline{a_g}a_h\varphi(g^{-1}h)\leq0
\end{eqnarray}
for any finite collection of coefficients $a_g\in {\Bbb C}$ with $\sum_g a_g=0$.
Sch\"oenberg's
theorem says that $$T_t: \lambda_g=e^{-t\varphi(g)}\lambda_g$$ extends to a  Markov semigroups of operators on the group von Neumann algebra ${  L}^\infty(\hat G)$ if and only if $\varphi$ is a  conditionally negative function with $\varphi(e)=0$. The negative generator of the semigroup is the unbounded map $$L:\la_g\mapsto \varphi(g)\la_g$$ which is weak* densely defined on ${  L}^\infty(\hat G)$.

Let $K_{\varphi}(g,h)=\frac12(\varphi(g)+\varphi(h)-\varphi(g^{-1}h))$, the Gromov form associated with $\varphi$. Then one can directly verify from (\ref{CN}) that $K_{\varphi}$ is a positive definite function on $G\times G$. Thus $K_{\varphi}^2$ is a positive definite function too. This is equivalent to the $\Gamma^2\geq0$ criterion for $T_t$, 
and therefore Theorem \ref{main3} applies to all such $(T_t)_t$'s. If in addition, $\varphi$ is real valued, then $(T_t)$ is a symmetric Markov semigroup. We then obtain the following corollary.

\begin{corollary}\label{last} Let $G$ be a locally compact unimodular group. Suppose $\varphi $ is  a  conditionally negative function on $G$ with $\varphi(e)=0$. Let  $L$ be the weak* densely defined linear map on ${  L}^\infty(\hat G)$ such that $L(\la_g)= \varphi(g)\la_g$. Then, 

(i) For any $\eta>\frac\pi2$ and any bounded analytic $\Phi$ on $S_\eta$, the map $\Phi(L): \la_g\mapsto \Phi(\varphi(g))\la_g$ extends to a completely bounded operator on BMO($\sqrt L)$ and $\|\Phi(L)\|\leq C_\eta\|\Phi\|_{\infty}$.

(ii) Suppose in addition that $\varphi$ is real valued. If $\Phi$ is a bounded analytic function on $S_\eta$ with $\eta>|\frac\pi2-\frac\pi p|$, then the map $\Phi(L)$ extends to a completely bounded operator on   ${  L}^p(\hat G)$ for $1<p<\infty$.
   
\end{corollary}

\noindent{\bf Remark.} Corollary \ref{last} (i) was proved in \cite{MS17} for $L: \la_g\mapsto  \sqrt {\varphi(g)}\la_g$ with  $\varphi$ a symmetric conditionally negative function on $G$.

\end{example}
 \begin{example} Let $G={\Bbb F}_\infty$ be the nonabelian free group with a countably infinite number of generators. Let $|g|$ be the reduced word length of $g\in G$. Then $\varphi:g\rightarrow |g|$ is a conditionally negative function (see \cite{Ha79}) and $L:\la_g\mapsto |g|\la_g$ generates a symmetric Markov semigroup on the free group von Neumann algebra. Fix $\theta\in(\frac\pi2,\pi)$, let $\Phi(z)=(\ln (z+2))^{-1}$ for $z\in S_\theta$. Then $\Phi\in H^\infty(S_\theta)$. 
  Corollary \ref{last} then implies that the Fourier multiplier 
 $$\la_g\mapsto \frac1{\ln (|g|+2)}\la_g $$ extends to a bounded operator on BMO$(\sqrt L)$. By the interpolation result Lemma \ref{JM12}, we conclude that this multiplier is bounded on $L^p(\hat {\Bbb F}_\infty)$ with constant $\lesssim \frac {p^2}{p-1}$.
 
 This produces a slowly decreasing multiplier which is bounded on $L^p(\hat {\Bbb F}_\infty)$ for all $1<p<\infty$. Note that Bo\.zejko and Fendler disproved the uniform $L^p$ boundedness of the $\ell_1$-length projections $P_N$, that map  $\la_g$ to $\chi_{\{g: |g|<N\}}\la_g$, for all $p>3$ (see \cite{BF06}). So  the classical method of producing slowing decreasing  $L^p$-multipliers through $P_N$ fails on free groups.   
 \end{example}
 \medskip
 

\medskip
\noindent {\bf Acknowledgement. } The second author is thankful to A. Mcintosh, P. Portal and L. Weis for helpful discussions and for introducing him the theory of bounded $H^\infty$-calculus.
   \bibliographystyle{amsplain}

\begin{thebibliography}{99}

\bibitem {ADM96} D. Albrecht, X. T. Duong, A. McIntosh, A. Operator theory and harmonic analysis. Instructional
Workshop on Analysis and Geometry, Part III (Canberra, 1995), 77-136. Proceedings of the
Center for Mathematical Analysis, Australian National University, 34. Australian National University,
Canberra, 1996.

\bibitem{Ba96}  D. Bakry, Remarques sur le semigroups de Jacobi. Asterisque \textbf{236} (1996), 23-39.

\bibitem{BBG12} D. Bakry, F. Bolley, I. Gentil,
Dimension dependent hypercontractivity for Gaussian kernels, Probability Theory and Related Fields, DOI: 10.1007/s00440-011-0387, arxiv:1003.5072.

\bibitem{BE85} D. Bakry and M. \'Emery, Diffusions hypercontractives. In S\'eminaire de probabilit\'es,
XIX, 1983/84, Lecture Notes in Math. 1123, pages 177-206. Springer, Berlin, 1985.

\bibitem{B60} A. V. Balakrishnan,  
Fractional powers of closed operators and the semigroups generated by them. 
Pacific J. Math. 10 (1960), 419-437. 

\bibitem{BF06} M. Bo\.zejko, G. Fendler,  A note on certain partial sum operators. Quantum probability, 117-125, Banach Center Publ., 73, Polish Acad. Sci. Inst. Math., Warsaw, 2006. 


\bibitem{BLM07} S. Bu,  C. Le Merdy,  
H$^p$-maximal regularity and operator valued multipliers on Hardy spaces.  
Canad. J. Math. 59 (2007), no. 6, 1207-1222. 

\bibitem{Co83}M. Cowling,   Harmonic analysis on semigroups. Ann. Math. 2 117 (1983), no. 2, 267-283.

\bibitem{CDMY96} M. Cowling, I. Doust, A. McIntosh, A. Yagi, Banach space operators with a bounded $H^\infty$-functional calculus. J. Austral. Math. Soc. Ser. A 60 (1996), no. 1, 51-89.

\bibitem{Da} Dabrowski Y., A non-commutative path space approach to stationary free stochastic differential equations, arXiv:1006.4351.

\bibitem{DDSY08} D. Deng, X. T. Duong, A. Sikora, L. X. Yan, 
Comparison of the classical BMO with the BMO spaces associated with operators and applications. (English summary) 
Rev. Mat. Iberoam. 24 (2008), no. 1, 267-296.

\bibitem {Du89} X. T. Duong,  $H^\infty$ functional calculus of second order elliptic partial differential operators on $L_p$-spaces, Proc. Centre Math. Appl. Austral. Nat. Univ., 24, Austral. Nat. Univ., (1989), 91-102.

\bibitem {DY05} X. T. Duong, L. Yan, Duality of Hardy and BMO spaces associated with
operators with heat kernel bounds. J. Amer. Math. Soc. 18 (2005), no. 4,
943--973.

\bibitem {Fe97} G. Fendler, Dilations of one parameter semigroups of positive contractions on $L^p$-spaces, {\it Canad. J. Math.}, 49(1997), 269-300.

\bibitem {GJL19} L. Gao, M. Junge, N. Laracuente, Fisher information and logarithmic Sobolev inequality for matrix-valued functions, arxiv1807.08838.

\bibitem {Ha06} M. Haase,  
The functional calculus for sectorial operators. 
Operator Theory: Advances and Applications, 169. Birkh\"auser Verlag, Basel, 2006.


\bibitem{Ha79} U. Haagerup,   An example of a nonnuclear C*-algebra, which has the metric approximation property. Invent. Math. 50 (1978/79), no. 3, 279-293.

\bibitem{HP98} M. Hieber, J. Pr\"uss, Funtional calculi for linear operators in vector valued $L^p$-spaces via the transference principle, Adv. Differential Equations, 3 (1998), 847-872.

\bibitem{HNP08}T. Hyt\"onen, J. van Neerven,  P. Portal,  
Conical square function estimates in UMD Banach spaces and applications to $H^\infty$-functional calculi,
J. Anal. Math. 106 (2008), 317-351.

\bibitem {JW17} Y. Jiao and M. Wang, Noncommutative harmonic analysis on semigroups.  Indiana Univ. Journal of Math, 66 (2017), no. 2, 401-417.

\bibitem{JRS} M. Junge, E. Ricard, D. Shlyakhtenko, Noncommutative diffusion semigroups and free probability, preprint.

 \bibitem{JM10} M. Junge, T. Mei,  Noncommutative Riesz transforms-a probabilistic approach. Amer. J. Math. 132 (2010), no. 3, 611-680.

\bibitem{JM12} M. Junge, T. Mei,  BMO spaces associated with semigroups of operators, Math. Ann. 352 (2012), no. 3, 691-743.

\bibitem{JMX06} M. Junge, C. Le Merdy,  Q. Xu,   $H^\infty$ functional calculus and square functions on noncommutative $L^p$-spaces. Ast\`{e}risque No. 305 (2006), vi+138 pp. 

\bibitem{KW04} P. Kunstmann, Peer C., L. Weis, Maximal $L^p$-regularity for parabolic equations, Fourier multiplier theorems and H$^\infty$-functional calculus. Functional analytic methods for evolution equations, 65-311, Lecture Notes in Math., 1855, Springer, Berlin, 2004.

\bibitem {Mc86} A. McIntosh,  
Operators which have an $H^\infty$  functional calculus. Miniconference on operator theory and partial differential equations (North Ryde, 1986), 210-231, 
Proc. Centre Math. Anal. Austral. Nat. Univ., 14, Austral. Nat. Univ., Canberra, 1986.

\bibitem{M08} T. Mei, Tent Spaces Associated with Semigroups of Operators, {\it Journal of Functional Analysis}, 255 (2008), 3356-3406.

\bibitem{MS17} T. Mei, M. de la Salle, Complete boundedness of heat semigroups on the von Neumann algebra of hyperbolic groups, 
Trans. Amer. Math. Soc., 369 (2017), no. 8, 5601-5622.

\bibitem{Mer12} C. Le Merdy,  
A sharp equivalence between $H^\infty$-functional calculus and square function estimates.  
J. Evol. Equ. 12 (2012), no. 4, 789-800.

\bibitem{MM07} G. Mauceri,  S. Meda,  
BMO and $H^1$ for the Ornstein-Uhlenbeck operator.  
J. Funct. Anal. 252 (2007), no. 1, 278-313.

\bibitem{PX03} G. Pisier, Q. Xu, Non-commutative $L_p$-spaces. Handbook of the
geometry of Banach spaces, Vol. 2, 1459-1517, North-Holland, Amsterdam,
2003.

\bibitem{R08} E. Ricard, A Markov dilation for self-adjoint Schur multipliers, Proc. Amer. Math. Soc. 136 (2008), no. 12, 4365-4372.


\bibitem{St70}  E. M. Stein, Topic in Harmonic Analysis (related to Littlewood-Paley
theory), Princeton Univ. Press, Princeton, New Jersey, 1970.

\bibitem{We01} L. Weis,  Operator-valued Fourier multiplier theorems and maximal $L^p$-regularity. Math. Ann. 319 (2001), no. 4, 735-758.

\bibitem{Y80} K. Yosida,  Functional analysis. Reprint of the sixth (1980) edition. Classics in Mathematics. Springer-Verlag, Berlin, 1995. xii+501 pp.
\end{thebibliography}

\bigskip
\hfill \noindent \textbf{Tim Ferguson} \\
\null \hfill Department of Mathematics
\\ \null \hfill   University of Alabama \\
\null \hfill Box 870350\\
\null \hfill Tuscaloosa, AL, 35487-0350 USA\\
\null \hfill\texttt{ tjferguson1@ua.edu}

\bigskip
\hfill \noindent \textbf{Tao Mei} \\
\null \hfill Department of Mathematics
\\ \null \hfill Baylor University \\
\null \hfill One bear place \#97328\\
\null \hfill  Waco, TX  76798, USA \\
\null \hfill\texttt{tao\_mei@baylor.edu}

\bigskip
\hfill \noindent \textbf{Brian Simanek} \\
\null \hfill Department of Mathematics
\\ \null \hfill Baylor University \\
\null \hfill One bear place \#97328\\
\null \hfill   Waco, TX  76798, USA \\
\null \hfill\texttt{brian\_simanek@baylor.edu}
\end{document}